\documentclass[reqno, 12pt, a4paper]{amsart}

\usepackage[truedimen,top=3truecm, bottom=3truecm, left=2.5truecm, right=2.5truecm, includefoot]{geometry}
\usepackage{fancyhdr}
\usepackage{abstract}

\usepackage{amsmath, amssymb, amsthm, mathtools} % Mathtools includes amsmath
\usepackage{mathrsfs}
\usepackage{enumerate}

\usepackage{graphicx}
\usepackage{float}
\usepackage{xcolor} % Load xcolor once with global options
\usepackage[percent]{overpic}
\usepackage{wrapfig} % para colocar texto em torno de figura

\definecolor{cyan(process)}{rgb}{0.0, 0.72, 0.92}
\definecolor{columbiablue}{rgb}{0.61, 0.87, 1.0}
\definecolor{sandstone}{HTML}{786D5F}
\definecolor{beaublue}{rgb}{0.74, 0.83, 0.9}
\definecolor{cherryblossompink}{rgb}{1.0, 0.72, 0.77}
\definecolor{light-gray}{gray}{0.95}

\usepackage{tikz}
\usepackage{pgfplots}
\pgfplotsset{compat=1.16}
\usepackage{pstricks,pst-node}
\usetikzlibrary{shapes.misc, arrows, decorations.pathmorphing, backgrounds, positioning, fit, petri, shapes, cd, arrows.meta, patterns, fadings, calc}

\usepackage[numbers]{natbib}
\usepackage[active]{srcltx} 
\usepackage{comment}
\usepackage{todonotes}
\usepackage[normalem]{ulem} % For strikethrough (\sout)

\usepackage[colorlinks=true, citecolor=blue, urlcolor=blue, pagebackref=true]{hyperref}
\usepackage{cleveref}
\usepackage{autonum} % to see only the references cited.

\renewcommand*{\backref}[1]{}
\renewcommand*{\backrefalt}[4]{%
    \ifcase #1 %
    \or
        \textcolor{red}{\textsuperscript{#2}}%
    \else
        \textcolor{red}{\textsuperscript{#2}}%
    \fi
}

\theoremstyle{plain}
\newtheorem{theorem}{Theorem}[section] 
\newtheorem{lemma}[theorem]{Lemma}
\newtheorem{corollary}[theorem]{Corollary}
\newtheorem{proposition}[theorem]{Proposition}
\theoremstyle{definition}
\newtheorem{definition}[theorem]{Definition}
\newtheorem{remark}[theorem]{Remark}

\begin{document}

\title[]{Spiking and Resetting}

\author{C\'edric Bernardin} 
\address{Faculty of Mathematics,National Research University Higher School of Economics, 6 Usacheva, Moscow, 119048, Russia}
\email{sedric.bernardin@gmail.com}

\author{Vsevolod Vladimirovich Tarsamaev}
\address{Faculty of Mathematics,National Research University Higher School of Economics, 6 Usacheva, Moscow, 119048, Russia}
\email{vvtarsamaev@edu.hse.ru}

\renewcommand{\subjclassname}{%
  \textup{2020} Mathematics Subject Classification}

\maketitle

% --------------------------------------------------------------------
% Abstract
\begin{abstract}
We consider a one-dimensional piecewise deterministic Markov process (PDMP) on $[0,1]$ with resetting at $0$ and depending on a small parameter $\varepsilon>0$. In the singular vanishing limit $\varepsilon \to 0$ we prove that the `` resetting '' simple point process associated to the PDMP converges to a point process described by a jump Markov process decorated by ``spikes'' distributed as a time-space Poisson point process with intensity proportional to $dt \otimes  x^{-2} dx$. This proves rigorously results appeared previously in \cite{SBDKC25} and also justifies partially a conjecture formulated there. 
\end{abstract}

% --------------------------------------------------------------------
%\newpage
%\tableofcontents
%\newpage

% --------------------------------------------------------------------

\section{Introduction, Model and Result}
\label{sec:imr}

The study of quantum trajectories \cite{BP02,WM10} is a field of significant physical interest, both from theoretical and practical perspectives, as evidenced by the Nobel Prizes awarded to S. Haroche and D.J. Wineland \cite{HW12} in 2012 and to J. Clarke, M.H. Devoret and J.M. Martinis \cite{CDM25} in 2025. Quantum trajectories can be viewed as scaling limits of discrete-time iterated quantum measurements \cite{AP06,P10} or as effective equations arising, for example, in the quantum filtering framework \cite{BVHJ09}. In the strong measurement regime, a quantum trajectory behaves as a pure jump Markov process on a finite set (pointer states). A complete mathematical understanding of the quantum jump phenomena observed in these experiments remains elusive, although significant progress has been made recently in the physics literature \cite{BBB12, BBT15} and in the mathematical literature \cite{BCFFS17, BBCCNP21, F22, FR24}.

More recently, the so-called ``spiking'' phenomenon was discovered, first heuristically in \cite{GP92,MW98,CBJP06,CBJ12}, and then treated theoretically by M. Bauer, D. Bernard and A. Tilloy in \cite{TBB15,BBT16,BB18}, which motivated further mathematical works, e.g. \cite{KL19,BCCNP23,BCNP25}.

Motivated initially by these physical questions, this paper is devoted to the rigorous study of the singular 'spiking' limit of certain one-dimensional piecewise deterministic Markov processes (PDMPs) on the state space $[0,1)$. These processes depend on a small parameter $\varepsilon>0$ that tends to zero, giving rise to a singular limit described by a space-time Poisson point process. We refer to \cite{SBDKC25} for physical motivations of the specific model studied here (see also \cite{DCD23}) and to \cite{BCCNP23} for mathematical studies considering similar problems in the context of singular limits of one-dimensional stochastic differential equations (SDEs) driven by Brownian white noise, in the strong noise limit regime. This paper presents the first mathematical study of singular ``spiking'' limits of one-dimensional SDEs driven by Poissonian white noise, in the strong noise limit regime. The questions addressed here could also find applications in the context of resetting problems \cite{EM11, EMS20,NG23}, but we will not discuss these in the current paper.

The aim of this paper is twofold: first, to provide a rigorous proof of the claims appearing in \cite{SBDKC25}, and second, to give a rigorous justification of a conjecture formulated in \cite{SBDKC25}. Although we do not completely justify the general conjecture presented in \cite{SBDKC25}, we prove it in the special ``resetting'' context, without relying on the explicit Laplace transform computations used in \cite{SBDKC25}.
\medskip 

More precisely, given a small parameter $\varepsilon>0$, the model is defined as follows. Let $f,h:[0,1] \mapsto \mathbb R$ be two times continuously differentiable functions such that
\begin{equation}
\label{eq:FH}
    f(1) < 0 < f(0) \quad  \text{and} \quad h(1) =0, \ h'(1) \ne 0,  \quad h(x)>0  \quad \text{if} \ x\in [0,1) \ .
\end{equation}
Since $h \vert_{[0,1)} > 0$, the condition $h(1)=0$ implies $h' (1)<0$. Let 
\begin{equation}
\omega^{\varepsilon}: x \in [0,1] \to \varepsilon f(x) + x h(x) \ ,
\end{equation}
and denote ${x}^\varepsilon =({x}^\varepsilon_t)_{t \ge 0}$  with state space $[0,1]$ the deterministic flow governed by the ordinary differential equation
\begin{equation}
\label{eq:deterministicflow}
\begin{split}
\varepsilon {\dot x^\varepsilon}_t = \omega^\varepsilon (x_t^\varepsilon) =\varepsilon f(x_t^\varepsilon) + x_t^\varepsilon h(x_t^\varepsilon), \quad x^\varepsilon_0=0 \ .
\end{split}
\end{equation}
We have that $\omega^\varepsilon (0)=f(0)>0$ and $\omega^\varepsilon (1)=f(1)<0$. We denote then
\begin{equation}
%\label{eq:}
\begin{split}
x^\varepsilon_*= \inf\{ x \in [0,1] \; ; \; \omega^\varepsilon (x)= 0 \}  \in (0,1) \ .
\end{split}
\end{equation}
The flow $x^\varepsilon$ is strictly increasing in time with 
\begin{equation}
%\label{eq:}
\begin{split}
\lim_{t \to \infty} x^\varepsilon_t =x^\varepsilon_* \ .
\end{split}
\end{equation}
Observe that because of the term $\varepsilon {\dot x^\varepsilon}_t$, Eq. \eqref{eq:deterministicflow} defines a flow presenting a time boundary layer at initial time{\footnote{For the reader not familiar with boundary layers, a simple example is provided by the solution of $\varepsilon \dot x_t^\varepsilon =-x_t^\varepsilon$, $x_0^\varepsilon=1$, whose solution is $x_t^\varepsilon=e^{-t/\varepsilon}$. For $t \ll \varepsilon$, $x_t^\varepsilon \approx 1$ while for $t \gg \varepsilon$, $x_t^\varepsilon \approx 0$. The main problem with the present boundary layer is then to describe the behaviour of $x_t^\varepsilon$ in the layer $ t \approx \varepsilon $ where simple Taylor expansions are not efficient and multiscales appear.}}. Since the flow is one-dimensional in space, it can be written more or less explicitly in some integral form, but resulting on some singular integrals (in space) defining it. This is illustrated in the ad hoc delicate asymptotic study performed in Section \ref{sec:nj-laplace-generating-convergence}.

\bigskip

We define now the (random) resetting dynamics ${X}^\varepsilon =( {X}^\varepsilon_t)_{t \ge 0}$  with state space $[0,1)$. Let $(\sigma_n^\varepsilon)_{n \ge 1}$ be a sequence of i.i.d. positive random variables such that 
\begin{equation}
\label{eq:mu_t_def}
\forall t > 0, \quad \mathbb P (\sigma^\varepsilon_n > t) := \mu_t^\varepsilon = \exp\left( -{\varepsilon}^{-1}  \int_0^t h(x^\varepsilon_s) \, ds \right) \ .
\end{equation}
Let us define $\tau^\varepsilon_0=0$ and for $n\ge 1$, 
\begin{equation}
\label{spike_times}
\tau^\varepsilon_n= \sigma^\varepsilon_1+ \ldots+\sigma^\varepsilon_n.
\end{equation} 
The resetting dynamics $X^\varepsilon$ is the c\`adl\`ag process defined for any $n \ge 0$ and $t \in [\tau^\varepsilon_n,\tau^\varepsilon_{n+1})$ by $X^\varepsilon_t = {x}^\varepsilon_t$. Observe that the sequence $(\tau^\varepsilon_n)_{n\ge 0}$ defines a renewal process on $(0,+\infty)$. The process is well defined since $\sup_{x \in [0,x^\varepsilon_*)} h(x) < \infty$ (no explosions). The generator $\mathcal L^\varepsilon$ of this process acts on differentiable test functions $f: [0,1] \to \mathbb R$ as 
\begin{equation}
\label{eq:generator}
(\mathcal L^\varepsilon f )(x) = {\varepsilon}^{-1}  \omega^\varepsilon (x)   f' (x) \ + \ {\varepsilon}^{-1} h(x)  \left[  f(0) - f(x) \right] \ . 
\end{equation}
The Markov process $X^\varepsilon$ satisfies the standard hypotheses given in \cite[p. 62, standard conditions 24.8]{Davies93} and is therefore a strong Markov process \cite[Theorem 25.5, p.64]{Davies93}. See Fig. \ref{fig:trajectory}.

\begin{figure}[htbp]
    \centering
    \label{fig:trajectory}

\begin{tikzpicture}[scale=1.5]
 % Axes
  \draw[->] (0,0) -- (9,0) node[right] {$t$};
  \draw[->] (0,0) -- (0,3) node[above] {$X_t^\varepsilon$};
  
 % Threshold y_*^? and asymptotic value x_*^? on left side
  \draw[dashed, green!70!black] (0,1.5) -- (9,1.5);
  \node[left, font=\small] at (0,1.5) {$y_*^\varepsilon$};
  \draw[dashed, gray] (0,2) -- (9,2);
  \node[left, font=\small] at (0,2) {$x_*^\varepsilon$};
  
  % First cycle: from 0 to e_2^?
  
  % First part (before e_1^?) - smooth increasing curves with resets
  \draw[thick, blue, smooth] plot coordinates {
    (0,0) (0.2,0.1) (0.4,0.25) (0.6,0.4) (0.8,0.5)
  };
  \draw[orange, thick, ->] (0.8,0.5) -- (0.8,0);
  \fill[orange] (0.8,0.5) circle (1.5pt);
  \node[below, font=\tiny] at (0.8,-0.15) {$\tau_1^\varepsilon$};
  
  \draw[thick, blue, smooth] plot coordinates {
    (0.8,0) (1.0,0.2) (1.2,0.5) (1.4,0.9) (1.5,1.1)
  };
  
  \draw[orange, thick, ->] (1.5,1.1) -- (1.5,0);
  \fill[orange] (1.5,1.1) circle (1.5pt);
  \node[below, font=\tiny] at (1.5,-0.15) {$\tau_2^\varepsilon$};
  
  \draw[thick, blue, smooth] plot coordinates {
    (1.5,0) (1.7,0.4) (1.9,0.9) (2.1,1.3) (2.3,1.5)
  };

  % Second part (between e_1^? and e_2^?) - smooth increasing curve
  \draw[thick, blue, smooth] plot coordinates {
    (2.3,1.5) (2.4,1.57) (2.7,1.64) (2.9,1.66) (3.1,1.68) 
    (3.3,1.70) (3.5,1.72) (3.7,1.74) (3.9,1.76) (4.1,1.78)
    (4.3,1.80) (4.5,1.82) (4.7,1.84) (4.9,1.86) (5.0,1.87)
  };

  % Mark e_1^? - first time above y_*^?
  \draw[blue, dashed, thick] (2.3,1.5)--(2.3,0);
 % \node[purple, above, font=\tiny] at (2.3,1.9) {$e_1^\varepsilon$};

 \fill[blue] (2.3,1.5) circle (1.5pt);
 \node[below, font=\tiny] at (2.3,-0.15) {$e_1^\varepsilon$};

  % Mark e_2^? - reset to 0
  \draw[orange, thick] (5.0,0) -- (5.0,1.0);
    \node[below, font=\tiny] at (5.0,-0.15) {$\tau_3^\varepsilon=e_2^\varepsilon$};
  \fill[orange] (5,1.87) circle (1.5pt);
  \draw[orange, thick, ->] (5.0,1.87) -- (5.0,0);
  
  % Second cycle: from e_2^? to e_4^?
  
  % First part of second cycle (before e_3^?)
  \draw[thick, blue, smooth] plot coordinates {
    (5.0,0) (5.2,0.2) (5.4,0.5) (5.6,0.8)
  }; 
  
  \draw[thick, blue, smooth] plot coordinates {
    (5.6,0) (5.8,0.3) (6.0,0.7) (6.2,1.2)
  }; 
  
  \draw[thick, blue, smooth] plot coordinates {
    (6.2,0) (6.4,0.6) (6.6,1.2) (6.8,1.5)
  };
  
 % Second part of second cycle (between e_3^? and e_4^?) - smooth increasing curve
  \draw[thick, blue, smooth] plot coordinates {
    (6.8,1.5) (7.0,1.62) (7.2,1.70) (7.4,1.76) (7.6,1.78)
    (7.8,1.80) (8.0,1.82) (8.2,1.84) (8.4,1.86) (8.6,1.88)
  };

 \draw[orange, thick, ->] (5.6,0.8) -- (5.6,0);
  \fill[orange] (5.6,0.8) circle (1.5pt);
  \node[below, font=\tiny] at (5.6,-0.15) {$\tau_4^\varepsilon$};

  \draw[orange, thick, ->] (6.2,1.2) -- (6.2,0);
  \fill[orange] (6.2,1.2) circle (1.5pt);
  \node[below, font=\tiny] at (6.2,-0.15) {$\tau_5^\varepsilon$};
  
  \fill[blue] (6.8,1.5) circle (1.5pt);
  \node[below, font=\tiny] at (6.8,-0.15) {$e_3^\varepsilon$};

  % Mark e_3^?
  \draw[blue, thick, dashed] (6.8,1.5)--(6.8,0);
  %\node[orange, above, font=\tiny] at (6.8,1.9) {$e_3^\varepsilon$};
  
    % Mark e_4^? - reset to 0
  \draw[orange, thick] (8.6,0) -- (8.6,1.0);
   \node[below, font=\tiny] at (8.6,-0.15) {$\tau_6^\varepsilon=e_4^\varepsilon$};
  \draw[orange, thick, ->] (8.6,1.88) -- (8.6,0);
  
   \fill[orange] (8.6,1.88) circle (1.5pt);

  % Extend axes for the additional part
  \draw[->] (8,0) -- (9,0) node[right] {$t$};

  % Legend (simplified)
  \draw[blue, thick] (0.5,2.8) -- (1.2,2.8);
  \node[right, font=\tiny] at (1.2,2.8) {$x_t^\varepsilon$};
  
  \draw[red, thick, ->] (3.0,2.8) -- (3.2,2.8);
  \node[right, font=\tiny] at (3.2,2.8) {Reset};
  
  \fill[orange] (4.7,2.8) circle (1.5pt);
  \node[right, font=\tiny] at (4.9,2.8) {Pre-spikes};
  
\end{tikzpicture}

\caption{A formal realisation of the PDMP $(X_t^\varepsilon)_{t \ge 0}$. }
\end{figure}
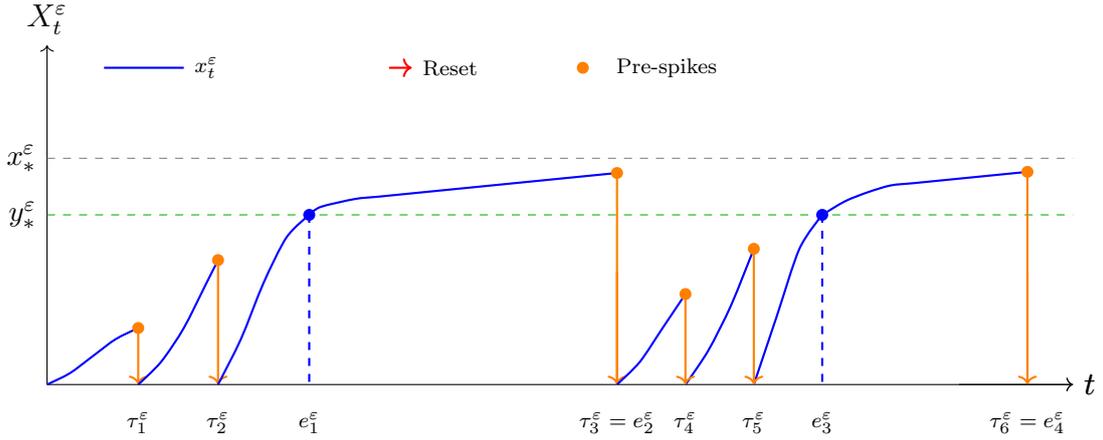

\bigskip

We are interested in the limit as $\varepsilon$ goes to zero of the time-space simple point process in $[0,\infty) \times [0,1]$  
\begin{equation}\label{measure_1}
    \mathbb Q^\varepsilon = \sum_{i \ge 1} \ \delta_{(\tau_i^\varepsilon,z_i^\varepsilon)} \ ,
\end{equation}
where the random points $(\tau_i^{\varepsilon}, z_i^\varepsilon):=(\tau^\varepsilon_i, X^\varepsilon_{{\tau^\varepsilon_i} \, -} )$, $i \ge 1$, are called the {\textit{pre-spikes}} and $\tau_i^\varepsilon$ the {\textit{pre-spikes times}}. It is conjectured in \cite{SBDKC25} that $(\mathbb Q^\varepsilon)_{\varepsilon>0}$ converges as $\varepsilon$ goes to zero to the decorated Poisson point process $\mathbb Q$ defined as follows.\\

Let $({\bar X}_t)_{t \ge 0}$ be the continuous pure jump Markov process on $\{0,1\}$ with rate $f(0)$ to jump from $0$ to $1$ and rate $f(1)$ to jump from $1$ to $0$. Let $(t,M_t)_{t \ge 0}$ be a time-space simple point process on $[0,\infty) \times [0,1]$ with support denoted $\mathbb M:=\{(t, M_t) \; ; \; t\ge 0\} \subset \mathbb [0,\infty)\times [0,1)$ and being defined as the Poisson point process with intensity measure
\begin{equation}
\label{full_int_meas}
\lambda_* = f(0)^2  \  \frac{dt \otimes dx}{x^2} \,{\bf 1}_{x \in [0,1]} 
\end{equation}
independent of ${\bar X}$. The points in $(t, \mathbb M_t) \in \mathbb M$ such that $\bar X_t=0$  are called spikes and are, in some sense to precize, the limit of the pre-spikes as $\varepsilon$ goes to zero. Note that the intensity measure is singular around $0$. However, it is not a problem of definition  if we restrict $\mathbb Q^\varepsilon$ (resp. $\mathbb Q$) to a given compact set in the form $[0,t] \times[\delta, 1]$ where $\delta \in (0,1]$, $t>0$. The point process $\mathbb Q$ is defined as
\begin{equation}
\label{eq:mathbbQ}
\mathbb Q = \sum_{(t, M_t) \in \mathbb M}  {\bf 1}_{{\bar X}_t = 0}\  \delta_{(t, M_t)} \ .
\end{equation}

Our main theorem is the following:

\begin{theorem}
\label{thm:mainone}
Let us fix $\delta \in (0,1]$ and $T>0$. We equip the space of Borel measures on the compact set $[0,T] \times [\delta, 1]$ with the weak topology.  We have that the restriction of the simple point process $(\mathbb Q^\varepsilon)_{\varepsilon>0}$ to $[0,T] \times [\delta, 1]$ converges in law to the restriction of the simple point process $\mathbb Q$ to $[0,T] \times [\delta, 1]$.
\end{theorem}

\begin{remark}
It remains quite mysterious that the intensity of the limiting Poisson point process appearing here is proportional to that appearing in \cite{BCCNP23}, which concerns SDEs driven by Brownian noises. In \cite{BCCNP23}, the intensity in the form $x^{-2} dx$ is clearly understood in the It\^o Brownian excursion theory framework. In the current paper, the authors do not understand why the same intensity appears again.
\end{remark}

\begin{remark}
The novelty of the paper with respect to \cite{SBDKC25} is twofold. First we generalise the analysis performed there by assuming quite generic drift and resetting terms, without hence relying to explicit Laplace transforms. Secondly, we justify rigorously Theorem \ref{thm:mainone} where several steps were missing in \cite{SBDKC25}.
\end{remark}

\bigskip

The paper is organized as follows. Section \ref{sec:convfiniterectangles} studies the convergence of the point process conditionally on the event of no jump. Section \ref{sec:proofofThm} then proves the main Theorem \ref{thm:mainone} by establishing the convergence of the point process to a Poisson point process. The key asymptotic analysis for the generating function is carried out in Section \ref{sec:nj-laplace-generating-convergence}, where Proposition \ref{prop:nj-laplace-generating-convergence} is proved. The paper also includes an appendix with several technical lemmas that support the main arguments.

\section{Rectangle convergence conditionally to no-jump}
\label{sec:convfiniterectangles}

For any $c\in (0,1)$, we define the deterministic time $T^\varepsilon_c$ as the hitting time to $c$ by the deterministic process $(x^\varepsilon_t)_{t \ge 0}$: 
\begin{equation}
%\label{eq:}
\begin{split}
T^\varepsilon_c = \inf \{ t \ge 0 \; ; \; x^\varepsilon_t =c\} \ .
\end{split}
\end{equation}
We also define the time
\begin{equation}
%\label{eq:}
\begin{split}
T_*^\varepsilon:= T^\varepsilon_{y^\varepsilon_*}, \quad y_*^\varepsilon = 1- \varepsilon^{\beta}
\end{split}
\end{equation}
where $\beta \in (0,1/2)$ is chosen arbitrarily. It is not difficult to show that as $\varepsilon$ goes to zero, we have that 
\begin{equation}
\label{eq:xstarepsilon}
\begin{split}
x^\varepsilon_*  \sim 1 -  \frac{f(1)}{\vert h'(1) \vert } \varepsilon \ .
\end{split}
\end{equation}
For $\varepsilon$ sufficiently small, we have then
\begin{equation}
%\label{eq:}
\begin{split}
y^\varepsilon_* \le x^\varepsilon_* \ ,
\end{split}
\end{equation}
and \begin{equation}
%\label{eq:}
\begin{split}
T_c^\varepsilon < \infty, \quad T^\varepsilon_{*} < \infty \quad \mathbb P \ \text{a.s.} \ .
\end{split}
\end{equation}
Roughly speaking the time $T_*^\varepsilon$ represents the time for the deterministic dynamics to travel from $0$ to $1$. The interested reader can check that for any $c \in (0,1)$ we have in fact that 
\begin{equation}
%\label{eq:}
\begin{split}
T^\varepsilon_c= -\frac{1}{h(0)} \varepsilon  \log \varepsilon  + \gamma (c) \varepsilon + o (\varepsilon) 
\end{split}
\end{equation}
where 
\begin{equation}
%\label{eq:}
\begin{split}
\gamma (c) = \frac{1}{h(0)}  \log \left(\frac{h(0)}{f(0)} \right) +  \int_{0}^c \left[ \frac{1}{y h(y)} -\frac{1}{h(0)y} \right] \, dy  +  \frac{1}{h(0)} \log c \ ,
\end{split}
\end{equation}
and hence
\begin{equation}
\label{eq:asymptoticsT_*}
\begin{split}
T^\varepsilon_*\sim_{\varepsilon \to 0}  -\frac{1}{h(0)} \varepsilon  \log \varepsilon \ .
\end{split}
\end{equation}
We define the stopping time $e_1^\varepsilon$ as{\footnote{Since $(y_*^\varepsilon, \infty)$ is an open set the random time $e_1^\varepsilon$ is indeed a stopping time.}}
\begin{equation}
%\label{eq:}
\begin{split}
e_1^\varepsilon:= \inf\{t \ge 0\; ; \; X_t^\varepsilon > y_*^ \varepsilon\} \ .
\end{split}
\end{equation}
Observe that since $X^\varepsilon$ has c\`adl\`ag trajectories, we have that $X^{\varepsilon}_{e_1^\varepsilon} = y_*^\varepsilon = \inf\{t \ge 0\; ; \; X_t^\varepsilon = y_*^ \varepsilon\} = \inf\{t \ge 0\; ; \; X_t^\varepsilon \ge y_*^ \varepsilon\}$. See Fig. \ref{fig:trajectory}.

%\begin{equation}
%%\label{eq:}
%\begin{split}
%e_0^\varepsilon :=\tau_{i_*^\varepsilon}^\varepsilon, \quad \text{where} \quad i_*^\varepsilon=\inf\{i \ge 0 \; ; \; \sigma_i^\varepsilon \ge T_*^{\varepsilon} \} \ .
%\end{split}
%\end{equation}

\begin{definition}
Given $0< t \le T$, we say that the process $X^\varepsilon$ does not jump from $0$ to $1$ during the time interval $(0,t)$ if and only if for any $r \in (0,t)$, we have $0 \le X_r^\varepsilon \le y_*^\varepsilon$, i.e. $e_1^\varepsilon \ge t$. This is equivalent to saying that for all $i\ge 1$ such that $\tau_i^\varepsilon \in (0,t)$, the condition $0<\tau^{\varepsilon}_{i+1} -\tau_i^\varepsilon<T_*^\varepsilon$ holds. 
\end{definition}

\bigskip

For any positive time $t>0$ and any sequence of $n \ge 1$ times $0<t_1< t_2< \ldots<t_n< t$, we denote by $p_t^\varepsilon (t_1,\ldots, t_n)$ the probability of observing no jumps and a sequence of exactly $n$ pre-spikes in the time interval $(0,t)$ occurring at times $t_1, \ldots,t_n$, i.e.
\begin{equation}
%\label{eq:}
\begin{split}
p_t^\varepsilon (t_1,\ldots, t_n) = \mathbb P \left( \{ \tau_1^\varepsilon=t_1, \ldots, \tau_{n}^\varepsilon = t_n\}  \, \cap \,  \{\tau_{n+1}^\varepsilon \ge t\} \,  \cap \,  \{ e_1^\varepsilon \ge t\} \right) \ .
\end{split}
\end{equation}
Let then $P^\varepsilon_{nj}(n,t : a,b)$ be the joint probability  to observe exactly $n$ pre-spikes in the time interval $(0,t)$ and in the space interval $[a,b]$ (where $0<a\le b\le1$), starting from $0$, and such that no jumps occur in the time interval $(0,t)$, i.e.  
\begin{equation}
%\label{eq:}
\begin{split}
&P^{\varepsilon}_{nj}(n,t : a,b) \\
&= \mathbb P \left( \{ \tau_1^\varepsilon=t_1, \ldots, \tau_{n}^\varepsilon = t_n\}  \, \cap \,  \{\tau_{n+1}^\varepsilon \ge t\} \, \cap \, \{ \forall i \in \{1,\ldots,n\}, \ z_{i}^\varepsilon \in [a,b] \} \,  \cap \,  \{ e_1^\varepsilon \ge t\} \right) \ . 
\end{split}
\end{equation}

Recall the definition of $\mu^\varepsilon$ in Eq. \eqref{eq:mu_t_def}. Following the combinatorial arguments in \cite{SBDKC25}, the probability $P_{nj} (n, t : a,b)$ is given by 
\begin{equation}
\label{eq:Pc}
\begin{split}
&P_{nj}^\varepsilon (n, t : a,b)\\
&= \sum_{m=0}^\infty \frac{(n+m)!}{n! m!} \prod_{i=1}^{n+m} \int_0^{t_{i+1}} dt_i  ~ p^\varepsilon_t (t_1,t_2,\ldots,t_{n+m}) \\
&\times \prod_{j=1}^m [~\Theta(T^\varepsilon_a- (t_{j}-t_{j-1}))  + \Theta((t_{j}-t_{j-1})-T^\varepsilon_b)~]\\ 
&\times \prod_{j=1}^n  \Theta((t_{j}-t_{j-1})-T^\varepsilon_a) ~ \Theta(T^\varepsilon_b-(t_{j}-t_{j-1}))\\
&\times \prod_{j=1}^{n+m} \Theta(T_*^\varepsilon- (t_{j}-t_{j-1}))~\Theta(T_*^\varepsilon-(t-t_{n+m})) \ , 
\end{split}
\end{equation}
where we have defined $t_{n+m+1}=t$ and $\Theta$ is the Heaviside function. Taking time-Laplace transform $\hat{P}_{nj}^\varepsilon (n:\sigma,a,b)= \int_{0}^{\infty} dt e^{-\sigma t} P_{nj}^{\varepsilon} (n,t : a,b)$ gives
\begin{equation}
\label{eq:Laplce-cond}
   \hat{P}_{nj}^\varepsilon (n, \sigma : a,b)= \sum_{m=0}^\infty \frac{(n+m)!}{n! m!} 
   [C^\varepsilon (\sigma)]^m [D^\varepsilon (\sigma)]^n E^\varepsilon(\sigma) \ ,
\end{equation}
where 
\begin{equation}
\label{eq:CDE}
\begin{split}
C^\varepsilon (\sigma)&=\varepsilon^{-1} \int_0^\infty dt ~\mu_t^\varepsilon\,  h (x_t^\varepsilon)~ [ \Theta(T_a^\varepsilon-t) + \Theta (t-T_b^\varepsilon)] ~ \Theta (T_*^\varepsilon-t) ~e^{-\sigma t}\ , \\ 
D^{\varepsilon} (\sigma)&=\varepsilon^{-1} \int_0^\infty dt ~\mu^\varepsilon_t \, h (x_t^\varepsilon) ~\Theta(T_a^\varepsilon-t) \Theta (t-T_b^\varepsilon)  e^{-\sigma t} \ , \\
E^{\varepsilon} (\sigma)&=\int_0^\infty dt~ \mu^\varepsilon_t~ \Theta(T_*^\varepsilon-t)~ e^{-\sigma t} \ .
\end{split}
\end{equation}
Performing the summation over $m$, we get
\begin{align}
   \hat{P}^\varepsilon_{nj}(n,\sigma
   : a,b)= \frac{  [D^\varepsilon(\sigma)]^n E^\varepsilon (\sigma) }{[1-C^\varepsilon(\sigma)]^{1+n}} \ . 
   \end{align}
Defining the generating function ($0\le z \le 1$)
\begin{equation}
Z^\varepsilon(z,\sigma :  a,b)=\sum_{n=0}^\infty z^n \hat{P}^\varepsilon_{nj}(n:\sigma,a,b) \ ,
\label{eq:Zc}
\end{equation}
we find the exact formula
\begin{align}
   \label{eq:functionZ}
   Z^\varepsilon (z, \sigma : a,b)= \frac{ E^\varepsilon(\sigma) }{1-C^\varepsilon (\sigma)-z D^\varepsilon (\sigma) } \ . 
   \end{align}
   
 \begin{proposition}
 \label{prop:nj-laplace-generating-convergence}
For any $\sigma \ge 0$, and $0<a<b\le1$, let us define the analytic function
\begin{equation}
%\label{eq:}
\begin{split}
Z(\cdot, \sigma : a,b): z \in D_{\sigma} (a,b) \mapsto Z(z,\sigma : a,b ):=\frac{1}{\sigma + f(0) + (1-z) f(0) (1/a -1/b)} \in \mathbb C 
\end{split}
\end{equation}
on the open disc $D_\sigma (a,b)$ of radius $R_\sigma (a,b)$, where
\begin{equation}
%\label{eq:}
\begin{split}
R_\sigma (a,b) = 1+\cfrac{\sigma + f(0)}{f(0) (1/a -1/b)} >1 \ .
\end{split}
\end{equation}
Then for any $z\in D_{\sigma} (a,b)$ we have 
 \begin{equation}
%\label{eq:}
\begin{split}
\lim_{\varepsilon \to 0} Z^\varepsilon (z,\sigma : a,b) = Z(z,\sigma : a,b ) \ .
\end{split}
\end{equation}
\end{proposition} 

\begin{proof}
This is proved in Section \ref{sec:nj-laplace-generating-convergence} by studying the asymptotic behaviour of the functions $C^\varepsilon(\sigma), D^{\varepsilon} (\sigma), E^{\varepsilon} (\sigma)$ defined in Eq. \eqref{eq:CDE}. In particular, there, we do not have to assume that $\sigma>0$ but only $\sigma\ge 0$. From this we have that the radius of convergence near $0$ of $Z^{\varepsilon} (\cdot, \sigma : a,b)$, which is equal to $[1-C^\varepsilon (\sigma)]/D^\varepsilon (\sigma)$, converges for small $\varepsilon$ to $R_\sigma (a,b)$.
\end{proof}

Observe that for $z=1$ and any $0<a\le b \le 1$, $Z^\varepsilon (1,\sigma : a,b)$ is the Laplace transform in time of the probability to not have jumps in the time interval $(0,t)$, i.e.{\footnote{In particular, the righthand side term does not depend on $a$ and $b$.}}
\begin{equation}
%\label{eq:}
\begin{split}
Z^\varepsilon (1, \sigma : a,b) =\int_0^\infty e^{-\sigma t} \, \mathbb P ( e_1^\varepsilon \ge t ) \, dt \ .
\end{split}
\end{equation}
Hence we get that
\begin{equation}
%\label{eq:}
\begin{split}
\lim_{\varepsilon \to 0} \int_0^\infty e^{-\sigma t} \, \mathbb P (e_1^\varepsilon \ge t ) \, dt =  \frac{1}{\sigma + f(0)} \ .
\end{split}
\end{equation}
This implies that

\begin{corollary}
\label{cor:e_1}
The sequence of random variables $(e_1^\varepsilon)_{\varepsilon>0}$ converges in distribution to an exponential random variable with parameter $f(0)$. 
\end{corollary}

We define now the probability $P_c (n, t : a,b)$ to observe exactly $n$ pre-spikes in the interval $(0,t)\times[a,b]$, {\textit{given}} that no jump occurs in the time interval $(0,t)$, i.e.
\begin{equation}
%\label{eq:}
\begin{split}
P^\varepsilon_c (n,  t : a,b):= \cfrac{P_{nj}^\varepsilon (n,t  : a,b)}{\mathbb P (e_1^\varepsilon \ge t)} \ .
\end{split}
\end{equation}

By Proposition \ref{prop:nj-laplace-generating-convergence} we get that
\begin{corollary}
\label{cor:Poisson0}
For any $t>0$ and any $0<a\le b\le1$, the probability distribution $P^\varepsilon_c( \cdot \, : \, t,a,b)$ over $\mathbb N_0$ converges in law to a Poisson distribution with parameter 
\begin{equation}
%\label{eq:}
\begin{split}
f(0) \, t\,  \lambda_* ( [a,b]) = f(0)\,  t \, \int_a^b \cfrac{dx}{x^2} = f(0) \, t \,  (a^{-1}- b^{-1}) \  . 
\end{split}
\end{equation}
Furthermore, we have that, for any $n\ge 0$, the function
\begin{equation}
%\label{eq:}
\begin{split}
\Phi_n^\varepsilon: r \in[0,\infty) \mapsto \mathbb P ( e_1^\varepsilon \ge r, N^\varepsilon_{(0,r]} ([a,b]) = n) \in [0,1]
\end{split}
\end{equation}
converges uniformly on any compact interval $[0,t]$, $t>0$, as $\varepsilon$ vanishes to 
\begin{equation}
%\label{eq:}
\begin{split}
\Phi_n: r \in [0,\infty) \mapsto e^{-r f(0)} e^{-r \lambda_* ([a,b]} \frac{(r \lambda_* ([a,b]))^n}{n!} \in [0,1] \ .
\end{split}
\end{equation}
\end{corollary} 

\begin{proof}
The first statement of the corollary follows straightforwardly from the second one. Hence, we prove now the second. Since $0 < a \le b\le1$ are fixed, we do not write the dependence in $a,b$ of the involved functions.  Fix $\sigma > 0$ and define for each $n \ge 0$:
\[
F^\varepsilon_\sigma(n) = \int_0^\infty e^{-\sigma s} \Phi_n^\varepsilon(s) \, ds, \quad
F_\sigma(n) = \int_0^\infty e^{-\sigma s} \Phi_n (s) \, ds \ .
\]
Then, for any $z \in \mathbb C$ such that $|z| < R_\sigma (a,b)$, for $\varepsilon$ sufficiently small, we have that
\[
Z^\varepsilon(z,\sigma) = \sum_{n=0}^\infty F^\varepsilon_\sigma(n) z^n, \quad
Z(z,\sigma) = \sum_{n=0}^\infty F_\sigma(n) z^n \ .
\]
The radius of  convergence of the power series $Z^\varepsilon (\cdot, \sigma)$ and $Z(\cdot,\sigma)$ are strictly bigger than one (uniformly as $\varepsilon$ goes to $0$). By Cauchy's formula for analytic functions we deduce that
\[
\lim_{\varepsilon \to 0} F^\varepsilon_\sigma(n) = F_\sigma(n) \ .
\]
Fix $n \in \mathbb N_0$ and observe that for any $s\le t$, we have 
\begin{equation}
%\label{eq:}
\begin{split}
\{ e_1^\varepsilon \ge t, \; N_{(0,t]}^\varepsilon \le n \} \subset \{ e_1^\varepsilon \ge s, \; N_{(0,s]}^\varepsilon \le n \} \ ,
\end{split}
\end{equation}
so that, for any $\varepsilon>0$, the (continuous)  function $h^{\varepsilon}: t \in [0, \infty) \mapsto h^\varepsilon (t) = \sum_{k=0}^n \Phi^\varepsilon_k (t)$, is  non-increasing, positive and bounded by $n$. The same holds for the function $h: t \in [0, \infty) \mapsto h^\varepsilon (t) = \sum_{k=0}^n \Phi_k (t)$. By Lemma \ref{lem:Tauber}, we have that $(h^\varepsilon)_{\varepsilon>0}$ converges uniformly to $h$ on every compact time interval. Hence, $(\Phi_n^\varepsilon)_{\varepsilon>0}$ also converges to $\Phi_n$ on any compact time interval.

%\bigskip
%
%Let us prove the third statement. Since the radius of convergence of $Z^\varepsilon (\cdot, \sigma)$ and $Z (\cdot, \sigma)$  are (uniformly in $\varepsilon$) strictly bigger than one, we have that
%\begin{equation}
%%\label{eq:}
%\begin{split}
%\lim_{\varepsilon \to 0} [Z^{\varepsilon}]' (1, \sigma) = Z' (1, \sigma) = \cfrac{f(0) (1/a -1/b)}{(\sigma + f(0))^2} \ .
%\end{split}
%\end{equation}
%On the other hand, we have that 
%\begin{equation}
%%\label{eq:}
%\begin{split}
%[Z^{\varepsilon}]' (1, \sigma) &= \sum_{n=0}^\infty n F_{\sigma}^\varepsilon (n) \\
%&= \int_0^\infty e^{-\sigma s} \left( \sum_{n=0}^\infty n \mathbb P ( e_1^\varepsilon \ge s, N^\varepsilon_{(0,s]} ([a,b]) = n)\right) \, ds\\
%&= \int_0^\infty e^{-\sigma s} \mathbb E \Big({\mathbf 1}_{e_1^\varepsilon \ge s} \, N^\varepsilon_{(0,s]} ([a,b])  \Big) \, ds \\
%&= \int_0^\infty e^{-\sigma s} \mathbb E \Big( N^\varepsilon_{(0,s]} ([a,b]) \,  \big\vert \, {e_1^\varepsilon \ge s}\Big)\, \mathbb P ({e_1^\varepsilon \ge s}) \, ds \ .
%\end{split}
%\end{equation}
%This is in particular true for $\sigma=0$.
\end{proof}

A similar property to Corollary \ref{cor:e_1} can be easily established. 

\begin{proposition}
\label{prop:e_2}
Let $e^\varepsilon_2$ be the stopping time defined by 
\begin{equation}
%\label{eq:}
\begin{split}
e^\varepsilon_2 = \inf \{ t \ge e_1^\varepsilon \; ; \; X_t^\varepsilon =0 \} \ .
\end{split}
\end{equation}
We have that
\begin{equation}
%\label{eq:}
\begin{split}
\lim_{\varepsilon \to 0} \mathbb P (e^\varepsilon_2 > t)  = e^{f(1) t}  \ ,
\end{split}
\end{equation}
i.e. $(e^\varepsilon_2)_{\varepsilon>0}$ converges in law to an exponential random variable with parameter $-f(1)$.
\end{proposition} 

\begin{proof}
Recall the definition of the stopping time
\[
e_2^\varepsilon = \inf\{ t \ge e_1^\varepsilon : X_t^\varepsilon = 0 \} \ .
\]
By the strong Markov property, conditionally to \( X_{e_1^\varepsilon}^\varepsilon = y_*^\varepsilon = 1-\varepsilon^\beta \) (with \(\beta\in(0,1/2)\)), the process evolves deterministically until the next reset. Let \(\tilde{x}_s^\varepsilon\) be the deterministic flow starting at \(y_*^\varepsilon\):
\[
\varepsilon \dot{\tilde{x}}_s^\varepsilon = \varepsilon f(\tilde{x}_s^\varepsilon) + \tilde{x}_s^\varepsilon h(\tilde{x}_s^\varepsilon), \qquad \tilde{x}_0^\varepsilon = 1-\varepsilon^\beta \  .
\]
Then
\[
\mathbb{P}(e_2^\varepsilon > t) = \exp\Bigl( -\int_0^t \frac{h(\tilde{x}_s^\varepsilon)}{\varepsilon} \, ds \Big) \ .
\]
We compute the integral by rewriting the ODE as
\[
\dot{\tilde{x}} = f(\tilde{x}) + \frac{\tilde{x} h(\tilde{x})}{\varepsilon} \ ,
\]
so that
\[
\frac{h(\tilde{x})}{\varepsilon} = \frac{1}{\tilde{x}} \bigl( \dot{\tilde{x}} - f(\tilde{x}) \bigr) = \frac{\dot{\tilde{x}}}{\tilde{x}} - \frac{f(\tilde{x})}{\tilde{x}}\  .
\]
Integrating from \(0\) to \(t\) gives
\begin{equation}
\label{eq:1ds}
\begin{split}
\int_0^t \frac{h(\tilde{x}_s^\varepsilon)}{\varepsilon} \, ds = \int_0^t \frac{\dot{\tilde{x}}_s^\varepsilon}{\tilde{x}_s^\varepsilon} \, ds - \int_0^t \frac{f(\tilde{x}_s^\varepsilon)}{\tilde{x}_s^\varepsilon} \, ds = \log\Bigl( \frac{\tilde{x}_t^\varepsilon}{\tilde{x}_0^\varepsilon} \Bigr) - \int_0^t \frac{f(\tilde{x}_s^\varepsilon)}{\tilde{x}_s^\varepsilon} \, ds \ . 
\end{split}
\end{equation}

\medskip
Now let \( v^\varepsilon (s) = 1 - \tilde{x}_s^\varepsilon \). Then \( v^\varepsilon (0) = \varepsilon^\beta \), and from the ODE we obtain
\[
\dot{v^\varepsilon} = -f(1-v^\varepsilon) - \frac{1}{\varepsilon}(1-v^\varepsilon)h(1-v^\varepsilon) \ .
\]
Since \( f(1)<0 \) and \( h(1)=0 \) with \( h'(1)<0 \), for small \( v>0 \) we have \( f(1-v)<0 \) and \( (1-v)h(1-v)>0 \). Hence, for sufficiently small \( \varepsilon \), the right--hand side is negative, so $v^\varepsilon$ is decreasing. Therefore, for all \( s \ge 0 \),
\[
0 \le v^\varepsilon (s) \le v^\varepsilon (0) = \varepsilon^\beta .
\]
Consequently,
\[
\sup_{s\in[0,t]} |1-\tilde{x}_s^\varepsilon| = \sup_{s\in[0,t]} v^\varepsilon (s) \le \varepsilon^\beta \xrightarrow[\varepsilon\to0]{} 0,
\]
so $\tilde{x}^\varepsilon$ converges uniformly on \([0,t]\) to the constant function $1$.

\medskip
From Eq. \eqref{eq:1ds}, since $\lim_{\varepsilon \to 0} \tilde{x}_t^\varepsilon =1$ and $ \lim_{\varepsilon \to 0} \tilde{x}_0^\varepsilon =\lim_{\varepsilon \to 0} ( 1-\varepsilon^\beta)=1$, we have
\[
\lim_{\varepsilon \to 0} \log\Bigl( \frac{\tilde{x}_t^\varepsilon}{1-\varepsilon^\beta} \Bigr)= \log 1 = 0 \ .
\]
Moreover, by uniform convergence, as $\varepsilon \to 0$,
\[
\frac{f(\tilde{x}_s^\varepsilon)}{\tilde{x}_s^\varepsilon} \longrightarrow f(1) \quad \text{uniformly on }[0,t],
\]
so
\[
\int_0^t \frac{f(\tilde{x}_s^\varepsilon)}{\tilde{x}_s^\varepsilon} \, ds \longrightarrow f(1)t \ .
\]
Thus,
\[
\lim_{\varepsilon\to0} \int_0^t \frac{h(\tilde{x}_s^\varepsilon)}{\varepsilon} \, ds = -f(1)t \ .
\]
Finally,
\[
\lim_{\varepsilon\to0} \mathbb{P}(e_2^\varepsilon > t) = \exp\bigl( -(-f(1)t) \bigr) = \exp\bigl( f(1)t \bigr) \ ,
\]
which means that \( e_2^\varepsilon \) converges in distribution to an exponential random variable with parameter \( -f(1) > 0 \).

\end{proof}

\section{Convergence to a Poisson process: proof of Theorem \ref{thm:mainone}}
\label{sec:proofofThm}

We denote by $(\mathcal F_t^\varepsilon)_{t \ge 0}$ the completed (by negligible sets) and augmented filtration of the natural filtration $(\sigma (X_s^\varepsilon \; ; \; 0 \le s \le t) )_{t \ge 0} $associated to $X^\varepsilon$, i.e. 
$$\mathcal F_t^\varepsilon = \sigma \Big( \cap_{ \delta>0} \sigma ( X_s^\varepsilon \; ; \; s \le t+\delta) \, \cup \,  \mathcal N\Big)$$ 
where $\mathcal N = \{ A \in \mathcal F \; ; \; \mathbb P (A) =0\}$. Let us then introduce a sequence of stopping times (with respect to $(\mathcal F_t^\varepsilon)_{t \ge 0}$ defined inductively by $e_0^\varepsilon=0$ and for any $k\ge 0$,
\begin{equation}
%\label{eq:}
\begin{split}
&e_{2k+1}^\varepsilon = \inf \{ t \ge e_{2k}^\varepsilon  \; ; \; X_t^\varepsilon >y_*^\varepsilon \} = \inf \{ t \ge e_{2k}^\varepsilon  \; ; \; X_t^\varepsilon >y_*^\varepsilon \} \ ,  \\  
&e_{2k+2}^{\varepsilon} = \inf \{ t \ge e_{2k+1}^\varepsilon  \; ; \; X_t^\varepsilon = 0 \} \ . 
\end{split}
\end{equation} 
Observe that by Strong Markov property, for any $\varepsilon>0$, the sequence $(e_k^\varepsilon)_{k \ge 0}$ is composed of independent random variables such that for any $k\ge 0$, $(e_{2k+1}^\varepsilon, e_{2k+2}^\varepsilon) = (e_1^\varepsilon, e_2^\varepsilon)$ in law. See Fig. \ref{fig:trajectory}. By Corollary \ref{cor:e_1} and Proposition \ref{prop:e_2}, we get the following result:

\begin{proposition}
The process $({\bar X}^\varepsilon_t)_{t \ge 0}$ with state space $\{0,1\}$ and c\`adl\`ag  trajectories defined by 
\begin{equation}
%\label{eq:}
\begin{split}
\forall t \ge 0, \quad {\bar X}^\varepsilon_t = \sum_{k\ge 0} {\mathbf 1}_{ t \in [e_{2k+1}^{\varepsilon}, e_{2k+2}^\varepsilon )} 
\end{split}
\end{equation}
converges in law to the jump Markov process $({\bar X}_t)_{t \ge 0}$ defined in Section \ref{sec:imr}.
\end{proposition}

Observe that this convergence is not sufficiently precise to see the convergence of the pre-spikes but only the convergence of the process of (quantum) jumps.

\bigskip

Denote by $\mathbb Q_k^{\varepsilon}$ the restriction of $\mathbb Q^\varepsilon$ to the time interval $[e_k^\varepsilon, e_{k+1}^\varepsilon)$, i.e. 
\begin{equation}
%\label{eq:}
\begin{split}
\forall A \in \mathcal B ([0, \infty) \times [0,1]), \quad \mathbb Q^\varepsilon_k (A)= \mathbb Q^\varepsilon \Big(A\cap \{ [e_k^\varepsilon, e_{k+1}^\varepsilon)\times [0,1]\} \Big) \ .
\end{split}
\end{equation}
We have that 
\begin{equation}
%\label{eq:}
\begin{split}
\mathbb Q^\varepsilon = \sum_{k \ge 0} \mathbb Q_k^\varepsilon \ .
\end{split}
\end{equation}
By strong Markov property, the sequence $(\mathbb Q_k)_{k \ge 0}$ is composed of independent random point processes. Moreover $(\mathbb Q_{2k})_{k\ge 0}$ (resp. $(\mathbb Q_{2k+1})_{k \ge 0}$) are identically distributed. Hence it is sufficient to study the limit of $(\mathbb Q_0^\varepsilon)_{\varepsilon >0}$ and $(\mathbb Q_1^\varepsilon)_{\varepsilon>0}$ as $\varepsilon$ goes to zero. The limit of $(\mathbb Q_1^\varepsilon)_{\varepsilon>0}$ is trivial in the sense that it is constant equal to zero. It remains only to study the limit of $(\mathbb Q_0^\varepsilon)_{\varepsilon >0}$. 

%\medskip
%
%For any bounded test function $u:[0,\infty) \times [0,1] \to \mathbb R$ with compact support included in $[0, \infty)\times (0,1]$ we have
%\begin{equation}
%%\label{eq:}
%\begin{split}
%\mathbb E \left[ \int u (\xi) d \mathbb Q_0^\varepsilon (\xi) \right] &= \int_{0}^\infty \mathbb E \left[ \int u (\xi) d \mathbb Q_0^\varepsilon (\xi) \, \Big\vert e_1^\varepsilon =t \right] \, d\mathbb P_{e_1^{\varepsilon}} (t) \\
%&= \sum_{i \ge 1} \int_{0}^\infty  \mathbb E \left[ u (\tau_i^\varepsilon, z_i^\varepsilon) {\bf 1}_{\tau_i^\varepsilon \le e_1^\varepsilon} \Big\vert e_1^\varepsilon =t  \right]  \, \, d\mathbb P_{e_1^{\varepsilon}} (t) 
%\end{split}
%\end{equation}
%

%
%\begin{lemma}
%The sequence $(\mathbb Q_0^\varepsilon)_{\varepsilon>0}$ is tight in the space of point processes on $\Omega:=[0,\infty)\times[0,1]$. 
%\end{lemma} 

%\begin{proof}
%Let $\delta>0$ and $T>0$ be fixed and denote $\Omega_{\delta, T} = [\delta, 1]\times [0,T]$.  The tightness is a trivial consequence of 
%\begin{equation}
%%\label{eq:}
%\begin{split}
%\limsup_{K \to \infty} \limsup_{\varepsilon \to 0} \mathbb P ( \big\vert \mathbb Q_0^\varepsilon (\Omega_{\delta, T}) \big\vert \ge K) = \limsup_{K \to \infty} \, e^{-\lambda_* (\Omega_{\delta,T})}  \sum_{n=K}^{\infty} \frac{\big[ \lambda_* (\Omega_{\delta,T})\big]^n}{n!} = 0 \ .
%\end{split}
%\end{equation}
%\end{proof}

For any times $0 \le s \le t$ and any Borel subset $A$ of $(0,1)$ we denote by $N^{\varepsilon}_{(s,t]} (A)$ the numbers of pre-spikes belonging to $(s,t]\times A$, i.e.
\begin{equation}
\label{eq:NstA}
\begin{split}
N_{(s,t]}^{\varepsilon} (A) := \Big\vert \big\{ (\tau_i^{\varepsilon}, z_i^\varepsilon)  \in (s,t] \times A \; ; \; i \in \mathbb N  \big\} \Big\vert \ ,
\end{split}
\end{equation}
 and to simplify notations
 \begin{equation}
%\label{eq:}
\begin{split}
N^{\varepsilon}_t (A):= N_{(0,t]}^\varepsilon (A) \ .
\end{split}
\end{equation}
Observe that 
\begin{equation}
\label{eq:NQ}
\begin{split}
N_{(s,t]}^{\varepsilon} (A)  = \big\vert \mathbb Q^\varepsilon ((s,t] \times A) \big\vert \ .
\end{split}
\end{equation}

\begin{proposition}
\label{prop:Poissonfinite}
Let $0<a\le b<1$ and $A=[a,b]$. For any $n\ge 0$ and any real numbers $0=s_0< s_1 < s_2 < s_n<t=s_{n+1}$, conditionally to the event $\{ e_1^\varepsilon \ge t\}$, the random vector
\begin{equation}
%\label{eq:}
\begin{split}
\big( N^\varepsilon_{(s_i, s_{i+1}]} (A)  \big)_{0 \le i \le n}
\end{split}
\end{equation}
converges in law, as $\varepsilon$ goes to $0$, to the random vector 
\begin{equation}
%\label{eq:}
\begin{split}
\big( {\bar N}_{i} \big)_{0 \le i \le n}
\end{split}
\end{equation}
composed of independent random variables such that for any $i \in \{0,\ldots, n\}$, $N_i$ is a Poisson random variable with parameter $(s_{i+1} -s_i) \lambda_* (A)$.
\end{proposition}

\begin{proof}
We prove the property by induction on $n$. By Corollary \ref{cor:Poisson0} we know that the proposition holds for $n=0$. 

\medskip

For any $\varepsilon>0$ and any $s>0$ we introduce the residual lifetime{\footnote{This terminology is inherited from renewal theory.}} $\xi_s^\varepsilon$ at $s$ defined by 
\begin{equation}
%\label{eq:}
\begin{split}
\xi_s^\varepsilon = \inf \{ \tau^\varepsilon_i \; ; \; i\ge 1, \ \tau^\varepsilon_i \ge s \} =  \inf \{ \tau^\varepsilon_i \; ; \; i\ge 1, \ \tau^\varepsilon_i > s \}
\end{split}
\end{equation}
which represents the first pre-spike time after time $s$. The second equality in the last display is proved by distinguishing the case where $s \in \{ \tau_i^\varepsilon \; ; \; i \ge 1\}$ and $s \notin \{ \tau_i^\varepsilon \; ; \; i \ge 1\}$. By lemma \ref{lem:stoppingtimexi}, we have that $\xi_s^\varepsilon$ is a $(\mathcal F_t^\varepsilon)_{t \ge 0}$ stopping time. Moreover we observe that
\begin{equation}
%\label{eq:}
\begin{split}
\mathbb P (\xi_s^\varepsilon = s) = \mathbb P ( \exists i \ge 1, \tau_i^\varepsilon =s) = \sum_{i \ge 1} \mathbb P (\tau_i^\varepsilon =s) =0
\end{split}
\end{equation}
where the last equality follows from the fact that $\tau_i^\varepsilon$ has a density with respect to the Lebesgue measure. Hence in the sequel we can always assume that $\xi_s^\varepsilon \ne s$.

\medskip
Assume that we have proved the induction hypothesis at level $n$, i.e. for any sequence of times $0=s_0<s_1< \ldots<s_n<s_{n+1}=t$, and let us prove it for a given sequence $0=s_0 < s_1< \ldots< s_{n+1}<s_{n+2}=t$. Let fix $k_0,k_1, \ldots, k_{n+1} \in \mathbb N_0$. To simplify notations, we write $N_\cdot^\varepsilon$ instead of $N_\cdot^\varepsilon (A)$. By the previous observations and strong Markov property we have that
\begin{equation}
%\label{eq:}
\begin{split}
&\mathbb P \Big (N_{(0,s_1]}^\varepsilon =k_0, N^\varepsilon_{(s_1,s_2]} = k_1, \ldots,  N^\varepsilon_{(s_{n+1},s_{n+2}]} = k_{n+1}, e_1^{\varepsilon} \ge t \Big)\\
&= \mathbb E \left( {\mathbf 1}_{ \{ N_{(0,s_1]}^\varepsilon =k_0, \ldots,  N^\varepsilon_{(s_{n},s_{n+1}]} = k_{n}, e_1^{\varepsilon} \ge s_{n+1} \}  } \; \mathbb E \Big[ {\mathbf 1}_{e_1^\varepsilon \ge t} {\mathbf 1}_{N^\varepsilon_{(s_{n+1},t]}= k_{n+1}} \vert {\mathcal F}^\varepsilon_{\xi_{s_{n+1}}^\varepsilon}  \Big] \right) \\
&= \mathbb E \left( {\mathbf 1}_{ \{ N_{(0,s_1]}^\varepsilon =k_0, \ldots,  N^\varepsilon_{(s_{n},s_{n+1}]} = k_{n}, e_1^{\varepsilon} \ge s_{n+1} \}  } \;  \Phi^\varepsilon (t-\xi_{s_{n+1}}^\varepsilon)\right)
\end{split}
\end{equation}
where 
\begin{equation}
%\label{eq:}
\begin{split}
\Phi^\varepsilon (r) =\mathbb P ( e_1^\varepsilon \ge r, N^\varepsilon_{(0,r]} = k_{n+1}) =\mathbb P ( N^\varepsilon_{(0,r]}= k_{n+1} \vert e_1^\varepsilon \ge r)\, \mathbb P (e_1^\varepsilon \ge r) \ . 
\end{split}
\end{equation}
We claim that 
\begin{equation}
\label{eq:claim222}
\begin{split}
\lim_{\varepsilon \to 0} \mathbb E  \left( {\bf 1}_{e_1^\varepsilon \ge s_{n+1}} \,  \vert \Phi^\varepsilon (t-\xi_{s_{n+1}}^\varepsilon) - \Phi (t -s_{n+1}) \vert \right) =0 \ ,
\end{split}
\end{equation}
where 
\begin{equation}
%\label{eq:}
\begin{split}
\Phi (r) = e^{- r f(0)} e^{- r \lambda_* (A)} \frac{[ r \lambda_* (A)]^{k_{n+1}}}{k_{n+1} !}\, e^{-r f(0)} \ .
\end{split}
\end{equation}
To prove Eq. \eqref{eq:claim222}, since $(\Phi^\varepsilon)_{\varepsilon>0}$ and $\Phi$ are uniformly bounded in absolute value by $1$, we can write, for any $\eta>0$, 
\begin{equation}
%\label{eq:}
\begin{split}
&\mathbb E  \left( {\bf 1}_{e_1^\varepsilon \ge s_{n+1}} \,  \vert \Phi^\varepsilon (t-\xi_{s_{n+1}}^\varepsilon) - \Phi (t -s_{n+1}) \vert \right) \\
& \le 2 \mathbb P  \left( e_1^\varepsilon \ge s_{n+1} \, , \,  \vert \xi_{s_{n+1}}^\varepsilon - s_{n+1} \vert \ge \eta  \right) \, + \, \sup_{s\in [0,t]} | \Phi^\varepsilon (s) - \Phi (s)| \, +\,  \omega( \Phi, \eta) \ , 
\end{split}
\end{equation}
where $\omega (f, \eta) = \sup_{|s-r| \le \eta} |f(s) -f(r)|$, for any function $f:[0,t] \to \mathbb R$. But, observe that by Corollary \ref{cor:e_1} and Corollary \ref{cor:Poisson0} we have that 
\begin{equation}
\label{eq:uniformPhi}
\begin{split}
\lim_{\varepsilon \to 0 } \Phi^\varepsilon (r) =\Phi(r) 
\end{split}
\end{equation}
uniformly in $r \in [0,t]$, and $\Phi$ is continuous so that, by taking the limsup in $\varepsilon \to 0$ and then in $\eta \to 0$, we get that 
\begin{equation}
%\label{eq:}
\begin{split}
&\limsup_{\varepsilon \to 0} \mathbb E  \left( {\bf 1}_{e^\varepsilon \ge s_{n+1}} \,  \vert \Phi^\varepsilon (t-\xi_{s_{n+1}}^\varepsilon) - \Phi (t -s_{n+1}) \vert \right)\\
& \le 2 \limsup_{\eta \to 0} \limsup_{\varepsilon \to 0}  \mathbb P  \left( e_1^\varepsilon \ge s_{n+1} \, , \,  \vert \xi_{s_{n+1}}^\varepsilon - s_{n+1} \vert \ge \eta  \right) \ .
\end{split}
\end{equation}
By Lemma  \ref{lem:rlt}, $\big({\bf 1}_{e^\varepsilon_1 \ge s_{n+1}} \, {\bf 1}_{ \xi_{s_{n+1}}^\varepsilon - s_{n+1} \, \ge \,  \eta} \big)_{\varepsilon>0}$ converges in probability to zero, so that Eq. \eqref{eq:claim222} follows. By Eq. \eqref{eq:claim222}, it follows then easily  
\begin{equation}
%\label{eq:}
\begin{split}
&\lim_{\varepsilon \to 0} \mathbb P \Big (N_{(0,s_1]}^\varepsilon =k_0, N^\varepsilon_{(s_1,s_2]} = k_1, \ldots,  N^\varepsilon_{(s_{n+1},s_{n+2}]} = k_{n+1}, e_1^{\varepsilon} \ge t \Big) \\ 
&= \Phi (t -s_{n+1})  \lim_{\varepsilon \to 0} \mathbb P \Big (N_{(0,s_1]}^\varepsilon =k_0, N^\varepsilon_{(s_1,s_2]} = k_1, \ldots,  N^\varepsilon_{(s_{n},s_{n+1}]} = k_{n}, e_1^{\varepsilon} \ge s_{n+1} \Big) \\
&=\Phi (t -s_{n+1})  \lim_{\varepsilon \to 0} \mathbb P \Big (N_{(0,s_1]}^\varepsilon =k_0, N^\varepsilon_{(s_1,s_2]} = k_1, \ldots,  N^\varepsilon_{(s_{n},s_{n+1}]} = k_{n} \vert e_1^{\varepsilon} \ge s_{n+1} \Big)\\
& \hspace{2,5cm}   \times \lim_{\varepsilon \to 0} \mathbb P (e_1^\varepsilon \ge s_{n+1}) \ .
\end{split}
\end{equation}
Then, by Corollary \ref{cor:e_1} and the induction hypothesis at level $n$, it follows that
\begin{equation}
%\label{eq:}
\begin{split}
&\lim_{\varepsilon \to 0} \mathbb P \Big (N_{(0,s_1]}^\varepsilon =k_1, N^\varepsilon_{(s_1,s_2]} = k_2, \ldots,  N^\varepsilon_{(s_{n+1},s_{n+2}]} = k_{n+2}, e_1^{\varepsilon} \ge t \Big) \\ 
&= e^{-t f(0)} \mathbb P ({\bar N}_0 = k_0) \ldots \mathbb P({\bar N}_{n}=k_{n}) \mathbb P ({\bar N}_{n+1} = k_{n+1}) \ ,
\end{split}
\end{equation}
where $(\bar N_0, \ldots, \bar N_{n+1})$ are $n+2$ independent random variables such that for any $i\in \{0,\ldots, n+1\}$, $\bar N_i$ is a Poisson variable with parameter $(s_{i+1} -s_i) \lambda_{*} (A)$. By Corollary \ref{cor:e_1} this proves the induction hypothesis at level $n+1$. 
\end{proof}

\begin{corollary}
\label{cor:indPoisson}
For any $A=[a,b] \subset (0,1]$ and any $t>0$, the sequence of simple point processes $(N^\varepsilon)_{\varepsilon>0}$ conditioned to the no-jump event $\{e_1^\varepsilon \ge t\}$ converges weakly in the Skorokhod space{\footnote{$D([0,t], \mathbb N_0)$ is the space of c\`adl\`ag functions from $[0,t]$ into $\mathbb N_0$, the latter being considered as a metric space equipped with the trivial distance.}} $D([0,t], \mathbb N_0)$ to a one dimensional Poisson process ${N} (A)$, with intensity $\lambda_* (A)$, restricted to the time interval $[0,t]$. 
\end{corollary}

\begin{proof}
By Proposition \ref{prop:Poissonfinite}, we already have that the finite time dimensional distributions of $(N^\varepsilon)_{\varepsilon>0}$ conditioned to the no-jump event $\{e_1^\varepsilon \ge t\}$ converge to the corresponding finite distributions of a Poisson process with intensity $\lambda_* (A)$. Hence, it remains only to establish that the sequence $(N^{\varepsilon} (A))_{\varepsilon>0}$ is tight in $D([0,t], \mathbb N_0)$ to conclude. Hence we have to show that
\begin{equation}
%\label{eq:}
\begin{split}
\limsup_{\varepsilon \to 0} \mathbb P (\text{number of jumps of $N^\varepsilon$ on $[0,t]$ $\ge n$  $\vert e_1^\varepsilon \ge t$} ) = 0  \ .
\end{split}
\end{equation}
 Recall Eq. \eqref{eq:NQ}. By Lemma \ref{lem:tight}, the result follows.

%Since $\Psi_k^\varepsilon (t) =\mathbb P (N^\varepsilon_{(0,t]} (A) = k \vert e_1^\varepsilon \ge t) \in [0,1]$ converges as $\varepsilon$ goes to $0$ to  $\Psi_k (t)= e^{-t \lambda_* (A)} \frac{(\lambda_* (A) t)^k}{k!} $ we have by Fatou lemma that
%\begin{equation}
%%\label{eq:}
%\begin{split}
%\limsup_{\varepsilon \to 0} \mathbb P (\text{number of jumps of $N^\varepsilon$ on $[0,t]$ $\ge n$  $\vert e_1^\varepsilon \ge t$} ) &= \limsup_{\varepsilon \to 0} \sum_{k \ge n} \Psi^\varepsilon_{k} (t)\\
%&\le \sum_{k \ge n} \Psi_k (t)  \ .
%\end{split}
%\end{equation}
%Observe then that $\lim_{n \to \infty} \sum_{k \ge n} \Psi_k (t) =0$ so that the tightness follows. 
\end{proof}

\begin{remark}
\label{rem:closedopen}
In the previous corollary, we choose $A=[a,b]$ but a similar statement holds if $A=[a,b)$ or $A=(a,b]$ or $A=(a,b)$ because, by Corollary \ref{cor:Poisson0}, we know that for every $n \ge 0$, 
\begin{equation}
%\label{eq:}
\begin{split}
\lim_{\varepsilon \to 0} \mathbb P (N^{\varepsilon}_t (\{a\}) = n ) = 0 \  . 
\end{split}
\end{equation}
\end{remark}

We recall now the following criterion of independency of Poisson processes .

\begin{proposition}{\cite[Ch. XII, Proposition 1.7]{RY13}}
\label{prop:RY}
Two Poisson processes $N^1$ and $N^2$ on $[0,t]$ are independent if and only if they do not jump simultaneously almost surely.
\end{proposition}

It follows that if $A=[a,b]$ and $C=[c,d]$ are two disjoint sub-intervals of $(0,1)$, then the two Poisson processes $N(A)$ and $N(C)$ defined through the Corollary \ref{cor:indPoisson} are independent. Indeed, assume for example that $a<b<c<d$ and define the interval $B=[a,d]$ which contains $A$ and $C$. By Corollary \ref{cor:indPoisson}, on $[0,t]$, $t>0$, $N(B)$ is a Poisson process with $\mathbb P$ a.s. a finite number of jumps (all of size one). We have that the jump times set of $N(B)$ contains the jump times set of $N(A)$ and $N(C)$. If $N(A)$ and $N(C)$ have a common jump time then $N(B)$ has a jump of size bigger than $2$, which is excluded.

\bigskip

Our aim is now to show that, for any $\delta \in (0,1)$ and $t>0$ fixed,  $(\mathbb Q^\varepsilon_0)_{\varepsilon>0}$, restricted to any time-space interval $[0,t] \times [\delta, 1]$,  converges weakly to a two-dimensional Poisson point process with intensity 
\begin{equation}
%\label{eq:}
\begin{split}
f(0)^2 \,  {\bf 1}_{s \in [0,t]} ds  \, {\bf 1}_{x \in [\delta,1]} \, \cfrac{dx}{x^2} \ .
\end{split}
\end{equation}
For simplicity of notation, we denote the restriction of $\mathbb Q^\varepsilon_0$ to $[0,t] \times [\delta, 1]$ by $\mathbb Q^\varepsilon_0$. Then the sequence $(\mathbb Q^\varepsilon_0)_{\varepsilon>0}$ is a family of random measures on $[0,t] \times [\delta, 1]$.

\bigskip 

Observe that for any $ A \in {\mathcal B} ([\delta, 1])$ and $0\le r\le s \le t$, we have 
\begin{equation}
\label{eq:additivity}
\begin{split}
\mathbb E \left[ \mathbb Q_0^\varepsilon ((r,s] \times A) \right] & = \mathbb E \left[ \sum_{i=1}^\infty {\bf 1}_{(\tau_i^\varepsilon, z_i^\varepsilon) \in (r,s] \times A} \, {\bf 1}_{\tau_i^\varepsilon \le e_1^\varepsilon} \right] \\
& =\mathbb E \left[ \sum_{i=1}^\infty {\bf 1}_{(\tau_i^\varepsilon, z_i^\varepsilon) \in (r,s] \times A} \, {\bf 1}_{\tau_i^\varepsilon \le e_1^\varepsilon} \, \Big \vert \, e_1^ \varepsilon \ge t \right] \ \mathbb P [e_1^\varepsilon \ge t ] \\
& =\mathbb E \left[ \sum_{i=1}^\infty {\bf 1}_{(\tau_i^\varepsilon, z_i^\varepsilon) \in (r,s] \times A} \, \Big \vert \, e_1^ \varepsilon \ge t \right] \ \mathbb P [e_1^\varepsilon \ge t ] \\
&=\mathbb E \left[ N^\varepsilon_{(r,s]} (A)  \, \Big \vert \, e_1^ \varepsilon \ge t \right] \ \mathbb P [e_1^\varepsilon \ge t ] \\
&= \mathbb E \left[ N^\varepsilon_{(r,s]} (A) \right] \ . 
\end{split}
\end{equation}

\begin{lemma}
\label{lem:tight}
For any $\delta>0$, there exists $K<\infty$ such that 
\begin{equation}
\label{eq:UI}
\begin{split}
\limsup_{\varepsilon \to 0} \mathbb E \left[ \mathbb Q_0^\varepsilon ([0,t] \times [\delta,1]) \right] \le K  \ .
\end{split}
\end{equation}
\end{lemma}
 
\begin{proof}
By Eq. \eqref{eq:additivity}, we have to prove 
\begin{equation}
%\label{eq:}
\begin{split}
\limsup_{\varepsilon \to 0} \mathbb E \left[ N^\varepsilon_{(0,t]} ([\delta,1])  \, \vert \, e_1^\varepsilon \ge t \right] \le K \ .
\end{split}
\end{equation}
We consider a new renewal process $(\tilde\tau_n^\varepsilon)_{n \ge 0}$ defined by $\tilde\tau_0^\varepsilon=0$ and for any $n \ge 1$, $\tilde\tau_n^\varepsilon = \tilde\sigma_1^\varepsilon + \ldots+\tilde\sigma_n^\varepsilon$ where the sequence of positive random variables $(\tilde\sigma_n^\varepsilon)_{n \ge 1}$ are i.i.d. with law defined by 
\begin{equation}
%\label{eq:}
\begin{split}
\mathbb P (\tilde\sigma^\varepsilon_1 \le s) = \mathbb P( \sigma_1^\varepsilon \le s  \, \vert \, \sigma_1^\varepsilon \le T_*^\varepsilon) = \cfrac{1}{\int_0^{T_*^\varepsilon}  \dot{\mu}_r^\varepsilon \, dr } \int_0^s {\bf 1}_{r \le T^\varepsilon_*}  \, \dot{\mu}_r^\varepsilon \, dr \ .
\end{split}
\end{equation}
For any $n \ge 1$, $\tilde\sigma_n^\varepsilon$ is stochastically smaller than $\sigma_n^\varepsilon$: $\mathbb P (\tilde\sigma_n^\varepsilon \le s) \le \mathbb P (\sigma_n^\varepsilon \le s)$ for any time $s \ge 0$. We define ${\tilde N}_{\cdot}^\varepsilon (\cdot)$ like in Eq. \eqref{eq:NstA} but,  replacing the renewal process $(\tau_n^\varepsilon)_{n\ge 0}$ by the renewal process $(\tilde\tau_n^\varepsilon)_{n \ge 0}$ and defining the PDMP $\tilde X^\varepsilon$ like $X^\varepsilon$ but with this new renewal sequence. We can couple the two processes so that for each $n$, $\tilde\sigma_n^\varepsilon \leq \sigma_n^\varepsilon $ $\mathbb P$ a.s. and consequently,
\begin{equation}
\label{eq:tau-tildetau}
\begin{split}
\tilde{\tau}^\varepsilon_n \leq \tau^\varepsilon_n \quad \text{for all } n \ .
\end{split}
\end{equation}
Since $x_s^\varepsilon$ is increasing in $s$, we have
\[
x_{\tilde\sigma_n}^\varepsilon \leq x_{\sigma_n}^\varepsilon.
\]
Conditional on the event $\{e_1^\varepsilon \geq t\}$, all inter-arrival times for the renewal process $(\tau_n^\varepsilon)_{n \ge 1}$, are smaller than $T_*^\varepsilon$, so the conditional law of $\sigma_n^\varepsilon$ given $\{e_1^\varepsilon \geq t\}$ is exactly the law of $\tilde\sigma_n^\varepsilon$. Therefore, conditionaly to this event, in law,
\[
N^\varepsilon_{(0,t]} ( [\delta,1]) = \tilde{N}^\varepsilon_{(0,t]}( [\delta,1]).
\]
It is therefore sufficient  to prove 
\begin{equation}
%\label{eq:}
\begin{split}
\limsup_{\varepsilon \to 0} \mathbb E  \left[ {\tilde N}^\varepsilon_{(0,t]} ([\delta,1]) \right] \le K \ .
\end{split}
\end{equation}
Since $\delta$ is fixed, to simplify notation, we denote $\tilde N^\varepsilon_{(0,t]} ([\delta,1])$ by $\tilde N_t^\varepsilon$.  

\bigskip 

We denote by $(\beta^\varepsilon_n)_{n \ge 1}$ the increasing subsequence of $(\tilde \tau^\varepsilon_n)_{n \ge 1}$ such that the corresponding pre-spike belong to $[\delta, 1]$, i.e. ${\tilde X}^\varepsilon_{\tilde\tau_\cdot^\varepsilon \, - } \in [\delta, 1]$. We denote by $(i^\varepsilon_n)_{n\ge 0}$ the (random) sequence of integers such that $ i_0^\varepsilon=0$ and 
\begin{equation}
%\label{eq:}
\begin{split}
\forall n \ge 1, \quad \beta_n^\varepsilon= \tilde \tau_{i_n^\varepsilon}^\varepsilon \ . 
\end{split}
\end{equation}
Consider the discrete-time process $H^\varepsilon$  given by 
\begin{equation}
%\label{eq:}
\begin{split}
\forall n \ge 0, \quad H^\varepsilon_n =  \sum_{k \ge 1} {\mathbf 1}_{i^\varepsilon_k \le n} \ . 
\end{split}
\end{equation}
See Fig. \ref{fig:trajectory-tightness}. Observe now that the random variable $H_n^\varepsilon$ has a Binomial law of parameter $(n,p^\varepsilon)$ where $p^\varepsilon \in [0,1]$ is the probability for ${\tilde X}^\varepsilon$ to reach $[\delta, 1)$ before be reseted to $0$. The discrete-time process $(M_n^\varepsilon)_{n \ge 0} := (H_n^\varepsilon - n p^\varepsilon)_{n \ge 0}$ is a centered discrete-time martingale with respect to the filtration $(\mathcal G^\varepsilon_n)_{n \ge 0}:= \big( \, \sigma ( {\tilde X}^\varepsilon_s \; ; \; s \le {\tilde \tau}_n^\varepsilon) \,  \big)_{n \ge 0}$. Consider the counting process $n^\varepsilon$ defined by 
\begin{equation}
%\label{eq:}
\begin{split}
\forall s \ge 0, \quad n_s^\varepsilon = \sup\{ n \ge 0 \; ; \; \beta_n^\varepsilon \le s\}= \sum_{n \ge 1} {\bf 1}_{{\beta}_n^\varepsilon \le s}  \ . 
\end{split}
\end{equation}
Observe that it is a stopping time so that, by the optional stopping theorem, we have
\begin{equation}
%\label{eq:}
\begin{split}
\forall s \ge 0, \quad \mathbb E (H^\varepsilon_{n_s^\varepsilon}) = p^\varepsilon \mathbb E (n_s^\varepsilon) \ .
\end{split}
\end{equation}
We have that
\begin{equation}
%\label{eq:}
\begin{split}
\mathbb E (\tilde N_t^\varepsilon) \le \mathbb E (H^\varepsilon_{n_t^\varepsilon}) = p^\varepsilon \mathbb E (n_t^\varepsilon) \ . 
\end{split}
\end{equation}
For any $r\ge 0$ we denote
\begin{equation}
%\label{eq:}
\begin{split}
\zeta_r^\varepsilon=\inf\{ \beta_i^\varepsilon \; ; \; \beta_i^\varepsilon \ge r\} \in [r, r+T_*^\varepsilon]  \ ,
\end{split}
\end{equation}
and we have thus
 \begin{equation}
\label{eq:subinterval}
\begin{split}
(\zeta_r^\varepsilon , \zeta_r^\varepsilon + s] \subset (r, s+r+T_*^\varepsilon] \ .
\end{split}
\end{equation}
For any $r,s \ge 0$, we define
\begin{equation}
%\label{eq:}
\begin{split}
n_{(r, r+s]}^\varepsilon = n_{r+s}^\varepsilon - n_r^\varepsilon \ . 
\end{split}
\end{equation}
Since, for any $r\ge 0$,  $\zeta_r^\varepsilon$ is a regenerative time (i.e. ${\tilde X}^\varepsilon_{\zeta_{r}^\varepsilon} =0$), we have that, in law, 
\begin{equation}
%\label{eq:}
\begin{split}
n_s^\varepsilon = n_{\big(\zeta_r^\varepsilon , \zeta_r^\varepsilon + s \big]}^\varepsilon \ . 
\end{split}
\end{equation}
Consequently, by Eq. \eqref{eq:subinterval}, we get that 
\begin{equation}
%\label{eq:}
\begin{split}
m^\varepsilon_s : = \mathbb E (n_s^\varepsilon) & = \mathbb E \left(n_{\big(\zeta_r^\varepsilon , \zeta_r^\varepsilon + s \big]}^\varepsilon \right) \le \mathbb E \left(n_{(r, s+r+T_*^\varepsilon]}^\varepsilon \right )   \le m^\varepsilon_{s+r+T_*^\varepsilon} - m^\varepsilon_r \ .
\end{split}
\end{equation}
By the renewal theorem, sending $r$ to infinity, we obtain then that
\begin{equation}
%\label{eq:}
\begin{split}
m^\varepsilon_s \le \frac{s+T_*^\varepsilon}{\mathbb E (\beta_1^\varepsilon)} \ .
\end{split}
\end{equation}

\bigskip

Therefore
\begin{equation}
%\label{eq:}
\begin{split}
\mathbb E (\tilde N_t^\varepsilon) \le p^\varepsilon \,  \frac{t+T_*^\varepsilon}{\mathbb E (\beta_1^\varepsilon)} \ .
\end{split}
\end{equation}

\bigskip 

By Eq. \eqref{eq:tau-tildetau} we have that 
\begin{equation}
%\label{eq:}
\begin{split}
p^\varepsilon&= \mathbb P ({\tilde \tau}_1^\varepsilon \ge T^\varepsilon_\delta) \le \mathbb P ({\tau}_1^\varepsilon \ge T^\varepsilon_\delta)=\mathbb P ({\sigma}_1^\varepsilon \ge T^\varepsilon_\delta)\\
& = \exp \left( - \varepsilon^{-1} \int_0^{T_\delta^\varepsilon} h(x_s^\varepsilon) \, ds \right)\\
&= \exp( -V^\varepsilon (\delta) ) \ .
\end{split}
\end{equation}
By Lemma \ref{lem:expV}, we have that $p^\varepsilon \lesssim \varepsilon$.

\bigskip

Observe now that 
\begin{equation}
%\label{eq:}
\begin{split}
\mathbb E (\beta_1^\varepsilon)= \cfrac{1}{-\int_0^{T_*^\varepsilon}  \dot{\mu}_r^\varepsilon \, dr } \int_0^{T^\varepsilon_*}  \, r \,  (-\dot{\mu}_r^\varepsilon)  \, dr \ .
\end{split}
\end{equation}
The numerator can be rewritten as
\begin{equation}
%\label{eq:}
\begin{split}
I^\varepsilon := \int_0^{T^\varepsilon_*}  \, r \,  (-\dot{\mu}_r^\varepsilon)  \, dr  =  \varepsilon  \int_0^{y_*^\varepsilon} \frac{h(x) }{\varepsilon f(x) + xh(x)} \, U^\varepsilon (x) \, e^{-V^\varepsilon (x)} \, dx  \ .
\end{split}
\end{equation}
By Lemma \ref{lem:Iepsilon} we have that it is bounded by below by $C \varepsilon$ where $C>0$ is a constant. Moreover, we have that the denominator can be written as
\begin{equation}
%\label{eq:}
\begin{split}
-\int_0^{T_*^\varepsilon}  \dot{\mu}_r^\varepsilon \, dr &= \int_0^{y_*^\varepsilon} e^{-V^\varepsilon (x)} \, \cfrac{h(x)}{\varepsilon f(x) +x h(x)} \,  dx \\
&= \int_0^{y_*^\varepsilon} \cfrac{d}{dx} \Big[ - e^{-V^\varepsilon (x)} \Big] \, dx \\
&=1- e^{-V^\varepsilon (y_*^\varepsilon)} =1 + o(1) \ .
\end{split}
\end{equation}
The last equality is proved in Lemma \ref{lem:Vyepsilonstar}. Hence we get the desired bound.

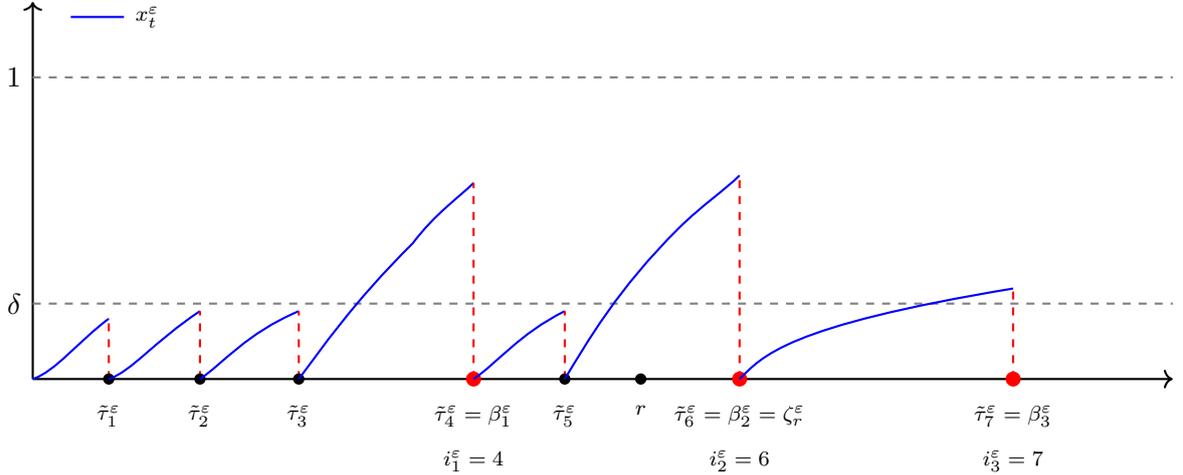
\begin{figure}[htbp]
    \centering
    \label{fig:trajectory-tightness}

\begin{tikzpicture}[
    line width=0.8pt,
    wave/.style={thick, blue},
    jump/.style={red, dashed, thick},
    annotation/.style={font=\small},
    timepoint/.style={circle, fill=black, inner sep=1.5pt},
    betapoint/.style={circle, fill=red, inner sep=2pt}
]

% Legend (simplified)
  \draw[blue, thick] (0.5,4.8) -- (1.2,4.8);
  \node[right, font=\tiny] at (1.2,4.8) {$x_t^\varepsilon$};

% Axes
\draw[->] (0,0) -- (15,0);
\draw[->] (0,0) -- (0,5);

% Space level delta and 1
\node[left, annotation] at (0,1) {$\delta$};
\draw[black!50, dashed] (0,1) -- (15,1);

\node[left, annotation] at (0,4) {$1$};
\draw[black!50, dashed] (0,4) -- (15,4);

% PDMP trajectory with multiple cycles
% Cycle 1
\draw[wave] (0,0) .. controls (0.3,0.1) and (0.6,0.5) .. (1,0.8);
\draw[jump] (1,0) -- (1,0.8);
\node[timepoint] at (1,0) {};
\node[below, annotation,font=\tiny] at (1,-0.2) {$\tilde{\tau}_1^\varepsilon$};

% Cycle 2
\draw[wave] (1,0) .. controls (1.3,0.1) and (1.6,0.5) .. (2.2,0.9);
\draw[jump] (2.2,0) -- (2.2,0.9);
\node[timepoint] at (2.2,0) {};
\node[below, annotation, font=\tiny] at (2.2,-0.2) {$\tilde{\tau}_2^\varepsilon$};

% Cycle 3
\draw[wave] (2.2,0) .. controls (2.5,0.2) and (2.8,0.6) .. (3.5,0.9);
\draw[jump] (3.5,0) -- (3.5,0.9);
\node[timepoint] at (3.5,0) {};
\node[below, annotation,font=\tiny] at (3.5,-0.2) {$\tilde{\tau}_3^\varepsilon$};

% Cycle 4 - beta point (above delta)
\draw[wave] (3.5,0) .. controls (3.8,0.4) and (4.2,1.0) .. (5,1.8)
            .. controls (5.3,2.2) and (5.6,2.4) .. (5.8,2.6);
\draw[jump] (5.8,0) -- (5.8,2.6);
\node[betapoint] at (5.8,0) {};
\node[below, annotation,font=\tiny] at (5.8,-0.2) {$\tilde{\tau}_4^\varepsilon = \beta_1^\varepsilon$};
\node[below, annotation,font=\tiny] at (5.8,-0.8) {$i_1^\varepsilon=4$};

% Cycle 5
\draw[wave] (5.8,0) .. controls (6.1,0.2) and (6.4,0.6) .. (7,0.9);
\draw[jump] (7,0) -- (7,0.9);
\node[timepoint] at (7,0) {};
\node[below, annotation,font=\tiny] at (7,-0.2) {$\tilde{\tau}_5^\varepsilon$};

%r and zeta_r
\node[timepoint] at (8,0) {};
\node[below, annotation,font=\tiny] at (8,-0.2) {$r$};

% Cycle 6 - beta point
\draw[wave] (7,0) .. controls (7.3,0.5) and (7.7,1.2) .. (8.5,2.0)
            .. controls (8.8,2.3) and (9.1,2.5) .. (9.3,2.7);
\draw[jump] (9.3,0) -- (9.3,2.7);
\node[betapoint] at (9.3,0) {};
\node[below, annotation,font=\tiny] at (9.3,-0.2) {$\tilde{\tau}_6^\varepsilon = \beta_2^\varepsilon=\zeta_r^\varepsilon$};
\node[below, annotation,font=\tiny] at (9.3,-0.8) {$i_2^\varepsilon=6$};

% Cycle 7
\draw[wave] (9.3,0) .. controls (9.6,0.3) and (9.9,0.7) .. (12.9,1.2);
\draw[jump] (12.9,0) -- (12.9,1.2);
\node[betapoint] at (12.9,0) {};
\node[below, annotation,font=\tiny] at (12.9,-0.2) {$\tilde{\tau}_7^\varepsilon =\beta_3^\varepsilon$};
\node[below, annotation,font=\tiny] at (12.9,-0.8) {$i_3^\varepsilon=7$};
\end{tikzpicture}

\caption{A formal realisation of the PDMP $({\tilde X}_t^\varepsilon)_{t \ge 0}$. }
\end{figure}

\end{proof}

By Markov inequality and Prokorhov theorem, this implies that the family of random measures $(\mathbb Q_0^\varepsilon)_{\varepsilon>0}$ is pre-compact in the space of measures on the compact set $[0,t] \times [\delta,1]$. Let $\mathbb Q^*$ be a limit point of the sequence $(\mathbb Q_0^\varepsilon)_{\varepsilon>0}$, which is a finite (a priori random) measure on $[0,t] \times [\delta,1]$. We need to show that $\mathbb Q^*$ is a two-dimensional Poisson point process with intensity 
\begin{equation}
%\label{eq:}
\begin{split}
f(0)^2 \,  {\bf 1}_{s \in [0,t]} ds  \, {\bf 1}_{x \in [\delta,1]} \, \cfrac{dx}{x^2} \ .
\end{split}
\end{equation}
Let $\mathbb Q$ be the law of such point process. To lighten notations we assume (otherwise we may extract a subsequence) that  $(\mathbb Q_0^\varepsilon)_{\varepsilon>0}$ converges (in distribution) to $\mathbb Q^*$. We need to show that $\mathbb Q^*=\mathbb Q$. By Lemma \ref{lem:tight}, we already know that $\mathbb Q^*$ is concentrated on the set of finite measures. By Corollary \ref{cor:indPoisson} and the comments following Proposition \ref{prop:RY}, $\mathbb Q^*$ is in fact the law of a simple point process.

\medskip

For any $t>0$ and $\delta \in (0,1]$, let $\mathcal U$ be the collection of sets which can be written as finite union of sets in the form $(s_i, t_i] \times A_i$, $[s_i,t_i]\times A_i$ or  $[s_i,t_i)\times A_i$, where $A_i \in \{ [a_i, b_i], (a_i, b_i], [a_i, b_i), (a_i, b_i) \}$, with $0 \le s_i <  t_i \le t$ and $\delta \le a_i \le b_i \le 1$.  By additivity property of Eq. \eqref{eq:additivity}, we have then for any $\Delta \in \mathcal U$,
\begin{equation}
%\label{eq:}
\begin{split}
\lim_{\varepsilon \to 0} \, \mathbb Q^\varepsilon (\Delta) = \mathbb Q^* (\Delta) = \mathbb Q(\Delta) \ .
\end{split}
\end{equation}
This is trivial if $\Delta$ is a finite \textit{disjoint} union of rectangles but we observe that if it is not the case, $\Delta$ can be rewritten as a disjoint union of rectangles (as explained above, boundaries do not have any importance). 

\medskip

Let us then recall 

\begin{theorem}\cite[R\'{e}nyi theorem]{DV02} 
\label{thm:renyi}
Let $\mu$ be a non-atomic measure on $[0,t]\times[\delta,1]$, finite on bounded sets. Suppose that the simple point process $\mathcal N$ is such that for any set $\Delta$ which is a finite union of rectangles
\[
P(\mathcal N(\Delta) = 0) = \exp\{-\mu(\Delta)\}.
\]
Then $\mathcal N$ is a Poisson process with intensity measure $\mu$.
\end{theorem}

Since $\lambda_*$ is non-atomic, this concludes the proof of  Theorem \ref{thm:mainone}.

\section{Proof of Proposition \ref{prop:nj-laplace-generating-convergence}}
\label{sec:nj-laplace-generating-convergence}

Let $0 <a \le b \le 1$. Recall Eq. \eqref{eq:Zc} which states that for any $\varepsilon>0$,  $\sigma \ge 0$ and $s \in[0,1]$,
\begin{equation}
%\label{eq:}
\begin{split}
Z^\varepsilon(s,\sigma: a,b)=\sum_{n=0}^\infty s^n \int_{0}^\infty e^{-\sigma t} {P}^\varepsilon_{nj}(n:t,a,b)\, dt 
\end{split}
\end{equation}
is given by Eq. \eqref{eq:functionZ}.

\bigskip

Recall the definition of $x_*^\varepsilon, y_*^\varepsilon$ and $T_c^\varepsilon, T_*^\varepsilon$ given at the begining of Section \ref{sec:convfiniterectangles}.

\bigskip

Let $U^\varepsilon:[0,x^\varepsilon_*) \mapsto [0,\infty)$ be the strictly increasing function defined by
\begin{equation}
\label{eq:Uepsilon}
\begin{split}
U^\varepsilon (x) = \int_{0}^x \cfrac{dy}{\varepsilon f(y) + y h(y)}, \quad x\in [0,x^\varepsilon_*) \ .
\end{split}
\end{equation} 
Since
\begin{equation}
%\label{eq:}
\begin{split}
\frac{d}{dt} U^\varepsilon (x^\varepsilon_t) = \big[U^\varepsilon\big]^{\prime} (x^\varepsilon_t) \, {\dot x}^\varepsilon_t = \varepsilon^{-1} \ ,
\end{split}
\end{equation}
we get that for any time $t\ge 0$, 
\begin{equation}
\label{eq:eq-sol-x}
\begin{split}
\varepsilon U^\varepsilon (x_t) = t \ . 
\end{split}
\end{equation}

Recall that $\mu^\varepsilon: t \in [0,\infty) \mapsto \mu_t^\varepsilon \in (0, 1]$ is the function defined by Eq. \eqref{eq:mu_t_def}, i.e. for any time $t\ge 0$,
\begin{equation}
\label{eq:mu}
\begin{split}
\mu_t^\varepsilon &= e^{- \varepsilon^{-1} \int_0^t  h(x^\varepsilon_s)ds} \ .
\end{split}
\end{equation}

Let $V^\varepsilon:[0, x^\varepsilon_*) \mapsto [0,\infty)$ be the strictly increasing function defined by 
\begin{equation}
\label{eq:Vepsilon}
\begin{split}
\forall x \in [0,x_\varepsilon^*), \quad V^\varepsilon (x)= \int_0^x \cfrac{h(y)}{\varepsilon f(y) + y h(y)} dy \ .
\end{split}
\end{equation}
We have that 
\begin{equation}
%\label{eq:}
\begin{split}
\frac{d}{dt} V^\varepsilon (x^\varepsilon_t) = \big[V^\varepsilon\big]^\prime  (x^\varepsilon_t) \,  {\dot x}^\varepsilon_t = \varepsilon^{-1} \cfrac{h(x^\varepsilon_t)}{\varepsilon f(x^\varepsilon_t) + x^\varepsilon_t h(x^\varepsilon_t)} \, [ \varepsilon f(x^\varepsilon_t) + x^\varepsilon_t h(x^\varepsilon_t)] = \varepsilon^{-1} h(x^\varepsilon_t) \ .
\end{split}
\end{equation}
It follows that
\begin{equation}
%\label{eq:}
\begin{split}
\varepsilon^{-1} \int_0^t h(x^\varepsilon_s) ds = V^\varepsilon (x^\varepsilon_t) 
\end{split}
\end{equation}
and consequently 
\begin{equation}
\label{eq:muV}
\begin{split}
\mu^\varepsilon_t = e^{- V^\varepsilon (x^\varepsilon_t)} \ .
\end{split}
\end{equation}

\bigskip

\subsection{Expansion of $E^\varepsilon (\sigma)$}

We recall that
\begin{equation}
%\label{eq:}
\begin{split}
E^\varepsilon(\sigma)= \int_0^{T_*^\varepsilon} e^{-\sigma t} \mu^\varepsilon_t dt \ . 
\end{split}
\end{equation}

\begin{lemma}
We have that as $\varepsilon$ goes to $0$, 
\begin{equation}
%\label{eq:}
\begin{split}
E^\varepsilon(\sigma) \sim \frac{1}{h(0)} \, \varepsilon \ .
\end{split}
\end{equation}
\end{lemma}

\begin{proof}

Using Eq. \eqref{eq:eq-sol-x} and Eq. \eqref{eq:muV} we have that
\begin{equation}
%\label{eq:}
\begin{split}
E^\varepsilon(\sigma)&=\varepsilon \int_0^{T_*^\varepsilon} \cfrac{e^{-\sigma \varepsilon U^\varepsilon (x^\varepsilon_t) - V^\varepsilon (x^\varepsilon_t) }}{\varepsilon f(x_t) +x^\varepsilon_t h(x^\varepsilon_t)} \ {\dot x^\varepsilon}_t \ dt = \varepsilon \int_0^{y^\varepsilon_*} \cfrac{e^{-\sigma \varepsilon U^\varepsilon (x) - V^\varepsilon (x) }}{\varepsilon f(x) +x h(x)} \, dx \ .
\end{split}
\end{equation}

\bigskip

For a fixed $\delta \in (0,x^*)$, chosen independently of $\varepsilon$ but sufficiently small, let us first evaluate the main contribution $E_\delta^\varepsilon(\sigma)$ to $E^\varepsilon (\sigma)$ given by 
\begin{equation}
%\label{eq:}
\begin{split}
E_\delta^\varepsilon(\sigma) &:=  \varepsilon \int_0^{\delta} \cfrac{e^{-\sigma \varepsilon U^\varepsilon (x) - V^\varepsilon (x) }}{\varepsilon f(x) +x h(x)} \, dx \\
&= \varepsilon \int_0^{\delta} \cfrac{e^{- V^\varepsilon (x) }}{\varepsilon f(x) +x h(x)} \, dx \ + \ \varepsilon \int_0^{\delta} \cfrac{ e^{- V^\varepsilon (x) }}{\varepsilon f(x) +x h(x)} \ (e^{-\sigma \varepsilon U^\varepsilon (x)} -1 ) \, dx\ .
\end{split}
\end{equation}
We have that for any $x \in [0,\delta]$,
\begin{equation}
%\label{eq:}
\begin{split}
\vert e^{-\sigma \varepsilon U^\varepsilon (x)} -1 \vert \le \sigma \varepsilon U^\varepsilon (x) \le \sigma \varepsilon U^\varepsilon (\delta) \ . 
\end{split}
\end{equation}
By Eq. \eqref{eq:0delta} and recalling the definition of $c(b)$ given in Eq. \eqref{eq:cdelta} we get that
\begin{equation}
%\label{eq:}
\begin{split}
\left\vert U^\varepsilon (\delta) - \int_0^\delta \frac{dy}{\varepsilon f(0) + y h(0)}\right\vert \le c^{-2}(\delta) \ \left[ \varepsilon \log (1+ \delta/\varepsilon) + \delta \right] \ .
\end{split}
\end{equation}
Hence, there exists a constant $C(\delta)>0$ depending only on $\delta$, such that
\begin{equation}
%\label{eq:}
\begin{split}
U^\varepsilon (\delta) \le C(\delta) \log (1/\varepsilon) \ .
\end{split}
\end{equation}
Thus, the integral 
\begin{equation}
%\label{eq:}
\begin{split}
\varepsilon \int_0^{\delta} \cfrac{ e^{- V^\varepsilon (x) }}{\varepsilon f(x) +x h(x)} \ (e^{-\sigma \varepsilon U^\varepsilon (x)} -1 ) \, dx
\end{split}
\end{equation}
is negligible with respect to 
\begin{equation}
%\label{eq:}
\begin{split}
 \varepsilon \int_0^{\delta} \cfrac{e^{- V^\varepsilon (x) }}{\varepsilon f(x) +x h(x)} \, dx \ .
\end{split}
\end{equation}
Therefore, we are reduced to find the asymptotic of the last integral. It can be rewritten as
\begin{equation}
%\label{eq:}
\begin{split}
\varepsilon \int_0^{\delta/\varepsilon} \cfrac{e^{- V^\varepsilon (\varepsilon x) }}{f(\varepsilon x) +x h(\varepsilon x)} \, dx 
\end{split}
\end{equation}
We claim that 
\begin{equation}
\label{eq:maincontribE}
\begin{split}
\varepsilon \int_0^{\delta/\varepsilon} \cfrac{e^{- V^\varepsilon (\varepsilon x) }}{f(\varepsilon x) +x h(\varepsilon x)} \, dx = \frac{\varepsilon}{h(0)}  \, + \, o(\varepsilon) \ .
\end{split}
\end{equation}

\bigskip 

To prove this claim, we observe first that 
\begin{equation}
%\label{eq:}
\begin{split}
V^\varepsilon (\varepsilon x) &= \int_0^x \cfrac{h(\varepsilon y)}{f(\varepsilon y) + y h(\varepsilon y)} \, dy \ .
\end{split}
\end{equation}

We have that for any $x \in [0,\delta/\varepsilon]$, 
\begin{equation}
%\label{eq:}
\begin{split}
& \int_{0}^x \left\vert \cfrac{h(\varepsilon y)}{f(\varepsilon y) + y h(\varepsilon y)} - \frac{h(0)}{f(0)+y h(0)} \right\vert \, dy \\
&= \int_0^x \left\vert \cfrac{ [ h(\varepsilon y) -h(0)] f(0) - h(0) [ f(\varepsilon y) -f(0)] }{[ f(\varepsilon y) + y h(\varepsilon y)] \, [ f(0)+y h(0)] } \right\vert \, dy \\
& \lesssim \varepsilon  \int_0^x  \cfrac{y}{[ f(\varepsilon y) + y h(\varepsilon y)] \, [ f(0)+y h(0)] } \, dy \\
& \lesssim  \kappa (\delta) \varepsilon  \int_0^x  \cfrac{y}{(1+y)^2} \, dy \ ,
\end{split}
\end{equation}
where 
\begin{equation}
%\label{eq:}
\begin{split}
1/ \kappa(\delta) := \inf\left\{ \inf_{z \in [0,\delta]} f(z) , \inf_{z \in [0,\delta]} h(z) \right\} > 0 \ .
\end{split}
\end{equation}
In the display above, the constant $\kappa (\delta)>0$ as soon as $\delta$ is taken sufficiently small. Observe now that
\begin{equation}
%\label{eq:}
\begin{split}
\int_0^x \frac{h(0)}{f(0) + y h(0) } \, dy =  \log \Big( 1+\tfrac{h(0)}{f(0)} x \Big) \ .
\end{split}
\end{equation}
It follows that
\begin{equation}
%\label{eq:}
\begin{split}
\left\vert e^{-V^\varepsilon (\varepsilon x)} - \frac{1}{1+\tfrac{h(0)}{f(0)} x} \right\vert & = \Big\vert e^{-\int_{0}^x \frac{h(\varepsilon y)}{f(\varepsilon y) + y h(\varepsilon y)} \, dy }  - e^{- \int_0^x \frac{h(0)}{f(0)+y h(0)} \, dy}  \Big\vert\\
& \le  \int_{0}^x \left\vert \frac{h(\varepsilon y)}{f(\varepsilon y) + y h(\varepsilon y)} - \frac{h(0)}{f(0)+y h(0)} \right\vert \, dy \\
& \lesssim  \kappa (\delta) \varepsilon  \int_0^x  \cfrac{y}{(1+y)^2} \, dy \ .
\end{split}
\end{equation}
We conclude that 
\begin{equation}
%\label{eq:}
\begin{split}
\varepsilon \int_0^{\delta/\varepsilon} \cfrac{e^{- V^\varepsilon (\varepsilon x) }}{f(\varepsilon x) +x h(\varepsilon x)} \, dx  & = \varepsilon \int_0^{\delta/\varepsilon} \frac{1}{1+ \tfrac{h(0)}{f(0)} x } \, \frac{1}{f(\varepsilon x) + x h( \varepsilon x)}\, dx \, + \, \theta_{\varepsilon}^\delta 
\end{split}
\end{equation}
where 
\begin{equation}
%\label{eq:}
\begin{split}
\vert \theta_\varepsilon^\delta \vert & \le \kappa (\delta) \varepsilon^2 \int_0^{\delta/\varepsilon} \frac{ \int_0^x y (1+y)^{-2}\, dy}{f(\varepsilon x) + x h(\varepsilon x)} \, dx \\
& \le  \kappa^2 (\delta) \varepsilon^2 \int_0^{\delta/\varepsilon} \frac{ \log(1+x)}{1 + x } \, dx  \\
&= \kappa^2 (\delta) \varepsilon^2 \log (1+ \delta/\varepsilon) \\
& = o(\varepsilon) \ .
\end{split}
\end{equation}
Hence, it remains to establish that 
\begin{equation}
%\label{eq:}
\begin{split}
\lim_{\varepsilon \to 0} \int_0^{\delta/\varepsilon} \frac{1}{1+ \tfrac{h(0)}{f(0)} x } \, \frac{1}{f(\varepsilon x) + x h( \varepsilon x)}\, dx = \frac{1}{h(0)} \ .
\end{split}
\end{equation}
The integrand converges pointwise to 
\begin{equation}
%\label{eq:}
\begin{split}
\frac{1}{1+ \tfrac{h(0)}{f(0)} x } \frac{1}{f(0) + x h(0)}  \ ,
\end{split}
\end{equation}
and we have 
\begin{equation}
%\label{eq:}
\begin{split}
& {\bf 1}_{x \le \delta/\varepsilon}  \frac{1}{1+ \tfrac{h(0)}{f(0)} x } \, \frac{1}{f(\varepsilon x) + x h( \varepsilon x)} \le \kappa(\delta) \frac{1}{\Big(1+ \tfrac{h(0)}{f(0)} x\Big) (1+x) } \ , \\
& \int_{0}^{\infty} \frac{1}{\Big(1+ \tfrac{h(0)}{f(0)} x\Big) (1+x) } \,  dx < \infty \ .
\end{split}
\end{equation}
By the dominated convergence theorem we get 
\begin{equation}
%\label{eq:}
\begin{split}
\lim_{\varepsilon \to 0} \int_0^{\delta/\varepsilon} \frac{1}{1+ \tfrac{h(0)}{f(0)} x } \, \frac{1}{f(\varepsilon x) + x h( \varepsilon x)}\, dx =   \int_0^ \infty \frac{1}{\Big(1+ \tfrac{h(0)}{f(0)} x\Big) (f(0)+x h(0)) } \, dx = \frac{1}{h(0)} \ .
\end{split}
\end{equation}
This proves Eq. \eqref{eq:maincontribE}.

\bigskip

Let us now show that the remaining contribution below is negligible with respect to $\varepsilon$:
\begin{equation}
\label{eq:Eepsilontail}
\begin{split}
\varepsilon \int_\delta^{y_*^\varepsilon} \cfrac{e^{-\sigma \varepsilon U^\varepsilon (x) - V^\varepsilon (x) }}{\varepsilon f(x) +x h(x)} \, dx = o(\varepsilon) \ .
\end{split}
\end{equation}

\bigskip

By using the same method as in the estimate of $D^\varepsilon (\sigma)$ we have that for any $0<\delta<b<1$, as $\varepsilon$ goes to $0$,
\begin{equation}
%\label{eq:}
\begin{split}
\varepsilon \int_\delta^{b} \cfrac{e^{-\sigma \varepsilon U^\varepsilon (x) - V^\varepsilon (x) }}{\varepsilon f(x) +x h(x)} \, dx = o(\varepsilon) \ . 
\end{split}
\end{equation}

\bigskip

Hence it remains to show that, for $b<1$, which can be chosen arbitrarily close to $1$, we have Eq. \eqref{eq:Eepsilontail}. We have that

\begin{equation}
%\label{eq:}
\begin{split}
\int_b^{y_*^\varepsilon} \cfrac{e^{-\sigma \varepsilon U^\varepsilon (x) - V^\varepsilon (x) }}{\varepsilon f(x) +x h(x)} \, dx \le \int_{b}^{y^\varepsilon_*} \cfrac{e^{- V^\varepsilon (x) }}{\varepsilon f(x) +x h(x)} \, dx 
\end{split}
\end{equation}
and by using Lemma \ref{lem:expV}, we can bound by above the last display by 
\begin{equation}
%\label{eq:}
\begin{split}
\varepsilon \, \tfrac{f(0)}{h(0)}\,   \int_{b}^{y_\varepsilon^*} \frac{1}{x}  \cfrac{e^{\vert R_{V^\varepsilon} (x)  \vert }}{\varepsilon f(x) +x h(x)} \, dx \lesssim \varepsilon \,   \int_{b}^{y^\varepsilon_*} \cfrac{e^{\vert R_{V^\varepsilon} (x)  \vert }}{\varepsilon f(x) +x h(x)} \, dx
\end{split}
\end{equation}
where 
\begin{equation}
%\label{eq:}
\begin{split}
\vert R_{V^\varepsilon} (x) \vert  \le C \big[ \varepsilon^{1/2 - \alpha}/ h(x) + \varepsilon^{2\alpha} h(x)  \big] \ .
\end{split}
\end{equation}
Above, $\alpha \in (0,1/2)$ is an arbitrary real number. Since $x \in [b,y^\varepsilon_*]$ we have that there exists a constant $c>0$ such that $h(x) \ge c (1-x)$,and consequently 
\begin{equation}
%\label{eq:}
\begin{split}
\forall x \in [b, y^*], \quad \vert R_{{V^\varepsilon}} (x) \vert \le C \big[ \tfrac{\varepsilon^{1/2 - \alpha}}{1-y^*} + \varepsilon^{2\alpha}  \big] \ .
\end{split}
\end{equation}
Since $1-y^\varepsilon_*=\varepsilon^{\beta}$, with the (optimal) choice $\alpha=1/6 - \beta/3$ we get that 
\begin{equation}
%\label{eq:}
\begin{split}
\forall x \in [b, y^\varepsilon_*], \quad \vert R_{V^\varepsilon} (x) \vert \le 2C \, \varepsilon^{1/3 - 2\beta/3} \ .
\end{split}
\end{equation}
Observe that since $\beta<1/2$ we have $1/3- 2\beta/3 >0$. It follows that
\begin{equation}
%\label{eq:}
\begin{split}
\varepsilon \int_b^{y^\varepsilon_*} \cfrac{e^{-\sigma \varepsilon U^\varepsilon (x) - V^\varepsilon (x) }}{\varepsilon f(x) +x h(x)} \, dx 
&\le \varepsilon \int_{b}^{y^\varepsilon_*} \cfrac{e^{- V^\varepsilon (x) }}{\varepsilon f(x) +x h(x)} \, dx \\
&\lesssim \varepsilon^{4/3 -2\beta/3} \Big[ U^\varepsilon(y^\varepsilon_*) -U^\varepsilon (b)\Big]\\
&\le  \varepsilon^{4/3 -2\beta/3} U^\varepsilon (y^\varepsilon_*) \ .
\end{split}
\end{equation}
By Lemma \ref{lem:expU}, observing that the upper bound of $\vert R_{U^\varepsilon} \vert $ is the same as the upper bound for $\vert R_{V^\varepsilon}\vert$, we have that 
\begin{equation}
%\label{eq:}
\begin{split}
\vert U^\varepsilon (y^\varepsilon_*) \vert  & \lesssim - \log \varepsilon + \varepsilon^{1/3 - 2\beta/3} \lesssim - \log \varepsilon \ , 
\end{split}
\end{equation}
so that
\begin{equation}
%\label{eq:}
\begin{split}
\varepsilon \int_b^{y^\varepsilon_*} \cfrac{e^{-\sigma \varepsilon U^\varepsilon (x) - V^\varepsilon (x) }}{\varepsilon f(x) +x h(x)} \, dx \lesssim - \varepsilon^{4/3 -2\beta/3} \log \varepsilon = o(\varepsilon) \ .
\end{split}
\end{equation}
This concludes the proof.
\end{proof}

\subsection{Expansion of $D^\varepsilon (\sigma)$}

For $0<a\le b<1$, the function $D^\varepsilon (\sigma)$ is defined for every $\sigma>0$ by
\begin{equation}
%\label{eq:}
\begin{split}
D^\varepsilon (\sigma)= \varepsilon^{-1} \int_{T_a^\varepsilon}^{T_b^\varepsilon} e^{-\sigma t} h(x^\varepsilon_t) \, \mu^\varepsilon_t \, dt \ .
\end{split}
\end{equation}
By using Eq. \eqref{eq:eq-sol-x} and Eq. \eqref{eq:muV} we have that
\begin{equation}
%\label{eq:}
\begin{split}
D^\varepsilon (\sigma)&= \varepsilon^{-1} \int_{T_a^\varepsilon}^{T_b^\varepsilon} e^{-\sigma t} h(x^\varepsilon_t) \mu^\varepsilon_t \, dt =  \varepsilon^{-1} \int_{T^\varepsilon_a}^{T_b^\varepsilon} e^{-\sigma \varepsilon U^\varepsilon (x^\varepsilon_t) -V^\varepsilon (x^\varepsilon_t)  } h(x^\varepsilon_t) \, dt \\
&=  \int_{T_a^\varepsilon}^{T_b^\varepsilon} e^{-\sigma \varepsilon U^\varepsilon (x^\varepsilon_t) -V^\varepsilon (x^\varepsilon_t)  } \frac{h(x^\varepsilon_t)}{\varepsilon f(x^\varepsilon_t) + x_t h(x^\varepsilon_t)} \, {\dot x}_t \,  dt \\
&= \int_a^b e^{- \sigma \varepsilon U^\varepsilon (x) - V^\varepsilon (x) }\frac{h(x)}{\varepsilon f(x) + x h(x)} \, dx \ .
\end{split}
\end{equation}
By Lemma \ref{lem:expU} and Lemma \ref{lem:expV}, taking in these lemmas $\alpha=1/6$, we have that 
\begin{equation}
%\label{eq:}
\begin{split}
- \sigma \varepsilon U^\varepsilon (x) - V^\varepsilon (x) &= \log \varepsilon -\log \frac{h(0)}{f(0)} - \log x + R_{\varepsilon} (x)
\end{split}
\end{equation}
where 
\begin{equation}
%\label{eq:}
\begin{split}
\sup_{x \in[a,b]} \vert R_\varepsilon (x) \vert \lesssim \varepsilon^{1/3} \ .
\end{split}
\end{equation}
Write
\begin{equation}
%\label{eq:}
\begin{split}
D^\varepsilon (\sigma)&= \varepsilon \, \frac{f(0)}{h(0)} \int_a^b \cfrac{1}{x} \cfrac{h(x)}{(\varepsilon f(x) +x h(x))} \, dx + \varepsilon \, \frac{f(0)}{h(0)} \int_a^b \big[e^{R_{\varepsilon} (x)} -1 \big] \cfrac{1}{x} \cfrac{h(x)}{(\varepsilon f(x) +x h(x))} \, dx \ .
\end{split}
\end{equation}
By using that $|e^z-1| \le |z| e^{|z|}$, we remark that
\begin{equation}
%\label{eq:}
\begin{split}
\left\vert \int_a^b \big[e^{R_{\varepsilon} (x)} -1 \big] \,  \cfrac{1}{x} \cfrac{h(x)}{(\varepsilon f(x) +x h(x))} \, dx \right\vert \lesssim \varepsilon^{1/3}  \int_a^b \cfrac{1}{x} \cfrac{h(x)}{(\varepsilon f(x) +x h(x))} \, dx
\end{split}
\end{equation}
so that  
\begin{equation}
\label{eq:D_varepsilon (sigma)}
\begin{split}
D^\varepsilon (\sigma) \sim  \varepsilon \, \frac{f(0)}{h(0)} \int_a^b \cfrac{1}{x} \cfrac{h(x)}{(\varepsilon f(x) +x h(x))} \, dx \sim \varepsilon \, \frac{f(0)}{h(0)} \int_a^b \cfrac{1}{x^2} \, dx = \varepsilon \, \frac{f(0)}{h(0)}\,  \left( \frac{1}{a} -\frac{1}{b} \right) \ .
\end{split}
\end{equation}

\subsection{Expansion of $C^\varepsilon (\sigma)$}

We have that
\begin{equation}
%\label{eq:}
\begin{split}
C^\varepsilon (\sigma) &= \varepsilon^{-1} \int_0^{T_*^\varepsilon} e^{-\sigma t} h(x^\varepsilon_t) \mu^\varepsilon_t \Big[ {\bf 1}_{t\le T_a^\varepsilon} +{\bf 1}_{t \ge T_b^\varepsilon} \Big] \,  dt \ .
\end{split}
\end{equation}
Since $0<a\le b <1$ are fixed we have that for $\varepsilon$ sufficiently small, $T_a^\varepsilon, T_b^\varepsilon \in (0, T_*^\varepsilon)$ and consequently
\begin{equation}
%\label{eq:}
\begin{split}
C^\varepsilon (\sigma)&= \varepsilon^{-1} \int_{0}^{T_a^\varepsilon} e^{-\sigma t} h(x^\varepsilon_t) \mu^\varepsilon_t \, dt \ + \ \varepsilon^{-1} \int_{T_b^\varepsilon}^{T_*^\varepsilon} e^{-\sigma t} h(x^\varepsilon_t) \mu^\varepsilon_t \, dt \\
&= \int_0^a e^{- \sigma \varepsilon U^\varepsilon (x) - V^\varepsilon (x) }\frac{h(x)}{\varepsilon f(x) + x h(x)} \, dx \ + \ \int_b^{y^\varepsilon_*} e^{- \sigma \varepsilon U^\varepsilon (x) - V^\varepsilon (x) }\frac{h(x)}{\varepsilon f(x) + x h(x)} \, dx \\
&= \int_0^{y^\varepsilon_*} e^{- \sigma \varepsilon U^\varepsilon (x) - V^\varepsilon (x) }\frac{h(x)}{\varepsilon f(x) + x h(x)} \, dx \ - \ D^\varepsilon (\sigma) \ .
\end{split}
\end{equation}
We rewrite $C^\varepsilon (\sigma)$ as
\begin{equation}
%\label{eq:}
\begin{split}
C^\varepsilon (\sigma)&= \int_0^{y^\varepsilon_*} e^{- \sigma \varepsilon U^\varepsilon (x) - V^\varepsilon (x) } \big[ \big[V^{\varepsilon}\big]^\prime (x) +\sigma \varepsilon \big[U^\varepsilon\big]^{'} (x)  \big] \, dx \ - \  \sigma \varepsilon \int_0^{y^\varepsilon_*} \cfrac{e^{-\sigma \varepsilon U^\varepsilon (x) - V^\varepsilon (x) }}{\varepsilon f(x) +x h(x)} \, dx - D^\varepsilon (\sigma)\\
&= - \int_0^{y^\varepsilon_*} \frac{d}{dx} e^{- \sigma \varepsilon U^\varepsilon (x) - V^\varepsilon (x) } \, dx \ - \  \sigma \varepsilon \int_0^{y^\varepsilon_*} \cfrac{e^{-\sigma \varepsilon U^\varepsilon (x) - V^\varepsilon (x) }}{\varepsilon f(x) +x h(x)} \, dx \ - \  D^\varepsilon (\sigma)\\
&= 1 -e^{- \sigma \varepsilon U^\varepsilon (y^\varepsilon_*) - V^\varepsilon (y^\varepsilon_*)} -\sigma E^\varepsilon (\sigma) - D^\varepsilon (\sigma)  \ .
\end{split}
\end{equation}

\bigskip 
We claim that, as $\varepsilon \to 0$, 
\begin{equation}
\label{eq:1-C}
\begin{split}
1- C^\varepsilon (\sigma)  \sim \varepsilon \, \frac{1}{h(0)} \, \Big[ \sigma + f(0) +  f(0) \Big( \frac{1}{a} -\frac{1}{b} \Big)\Big] \ .
\end{split}
\end{equation}
\bigskip
To prove  Eq. \eqref{eq:1-C}, it remains thus to show that 
\begin{equation}
\label{eq:hard0}
\begin{split}
\exp \Big\{- \sigma \varepsilon U^\varepsilon (y^\varepsilon_*) - V^\varepsilon (y^\varepsilon_*)\Big\} \sim \varepsilon \frac{f(0)}{h(0)} \ .
\end{split}
\end{equation}
By Lemma \ref{lem:expU} and Lemma \ref{lem:expV}, for any $\alpha \in (0,1/2)$, we have that
\begin{equation}
%\label{eq:}
\begin{split}
&\sigma \varepsilon U^\varepsilon(y^\varepsilon_*) + V^\varepsilon (y^\varepsilon_*) = \log \left( \frac{h(0)}{f(0)}\right) - \log \varepsilon + \Theta_\varepsilon
\end{split}
\end{equation}
 where 
 \begin{equation}
%\label{eq:}
\begin{split}
\Theta_\varepsilon &= - \frac{\sigma}{h(0)} \varepsilon \log\varepsilon + \frac{\sigma \varepsilon}{h(0)} \log  \left( \frac{h(0)}{f(0)}\right) + \sigma \varepsilon \int_0^{y^\varepsilon_*} \left[ \frac{1}{y h(y)} -\frac{1}{h(0) y} \right] \, dy\\
&  + \left( 1+ \frac{\sigma}{h(0)} \varepsilon \right)  \log y^\varepsilon_* + R_{U^\varepsilon} (y^\varepsilon_*)  +R_{V^\varepsilon} (y^\varepsilon_*) \ .
\end{split}
\end{equation}
Recall that $y^\varepsilon_* =1- \varepsilon^\beta$ with $\alpha=1/6 - \beta/3$, $\alpha, \beta \in (0,1/2)$. Then it is straightforward to check that $\lim_{\varepsilon \to 0} \Theta_\varepsilon=0$ because $1- \varepsilon^\beta \lesssim h(y^\varepsilon_*)$. This concludes the proof of \eqref{eq:hard0}.

%%%%%%%%%%%%%%%%%%%%%%%%%%%%%%%%%%%%%%%%%%%%%%%%%%%
%
%                                                     APPENDIX
%
%%%%%%%%%%%%%%%%%%%%%%%%%%%%%%%%%%%%%%%%%%%%%%%%%%%%%

\appendix

%%%%%%%%%%%%%%%%%%%%%%%%%%%%%%%%%%%%%

\section{Probabilistic technical lemmas}

%%%%%%%%%%%%%%%%%%%%%%%%%%%%%%%%%%%%%

\begin{lemma}
\label{lem:stoppingtimexi}
For any $\varepsilon>0$ and any $s>0$, $\xi_s^\varepsilon$ is an $(\mathcal F_t^\varepsilon)_{t \ge 0}$ stopping time.
\end{lemma} 

\begin{proof}
Let $t \ge 0$ be given. If $s>t$ then $\xi_s^\varepsilon =\infty$  and $\{ \xi_s^ \varepsilon \le t\} = \emptyset \in \mathcal F_t^\varepsilon$. If $s \le t$, we have 
\begin{equation}
%\label{eq:}
\begin{split}
\{ \xi_s^\varepsilon \le t\} =  \cup_{i \ge 1} \Big (\{ \tau_i^\varepsilon \le t\} \cap \{\tau_i^\varepsilon \ge s\} \Big) \ . 
\end{split}
\end{equation}
But $\{ \tau_i^\varepsilon \ge s\} = \cap_{k \ge 1} \{\tau_i^\varepsilon \le s-1/k\}^c \in \mathcal F^\varepsilon_s \subset \mathcal F_t^\varepsilon$ since $(\mathcal F_t^\varepsilon)_{t \ge 0}$ is increasing and $\tau_i^\varepsilon$ is a stopping time. Hence $\{ \xi_s^ \varepsilon \le t\} \in \mathcal F_t^\varepsilon$.
\end{proof}

\begin{lemma}
\label{lem:rlt}
For any $s>0$ and $\eta>0$, we have that 
\begin{equation}
%\label{eq:rlt}
\begin{split}
\Big( {\bf 1}_{e^\varepsilon_1 \ge s} \, {\bf 1}_{ \xi_{s}^\varepsilon - s \, \ge \,  \eta} \Big)_{\varepsilon>0} 
\end{split}
\end{equation}
converges, as $\varepsilon$ vanishes, to zero, in $\mathbb L^1$. 
\end{lemma}

\begin{proof}

The condition $e_1^\varepsilon \ge s$ implies, by definition of $e_1^\varepsilon$ and $\xi_s^\varepsilon$,  that $\xi_s^\varepsilon -s \le T_*^\varepsilon$. By Eq. \eqref{eq:asymptoticsT_*}, we have that $T_*^\varepsilon$ vanishes as $\varepsilon$ goes to $0$. The result follows.

\end{proof}

\section{Asymptotic of integrals and consequences}

%\section{Evaluation of $U^\varepsilon$}
%\label{app:expUvarepsilon}
%

For $x\in [0,1)$, let us define
\begin{equation}
\label{eq:h}
\begin{split}
{\mathfrak h}(x)= \inf_{y \in [0,x]} h(y) >0 
\end{split}
\end{equation}
and observe that $\lim_{x \to 1^-} {\mathfrak h}(x)=0$. For $\delta>0$, we also define 
\begin{equation}
\label{eq:cdelta}
\begin{split}
c:=c(\delta)=\inf\left\{  \inf_{y \in [0,\delta]} f(y)\, , \, \inf_{y \in[0,\delta]} h(y) \right\} \ .
\end{split}
\end{equation}
Observe that $\lim_{\delta \to 0} c(\delta) = \inf(f(0), h(0))>0$.

\begin{lemma}
\label{lem:expU}
There exists a constant $C>0$ such that for any $\alpha \in (0,1/2)$ and $x\in (0,1)$, the following holds
\begin{equation}
\label{eq:exp-Uvarepsilon}
\begin{split}
U^\varepsilon(x) &= -\frac{1}{h(0)} \log \varepsilon +  \frac{1}{h(0)}  \log \Big(\frac{h(0)}{f(0)} \Big) +  \int_{0}^x \left[ \frac{1}{y h(y)} -\frac{1}{h(0)y} \right] \, dy  +  \frac{1}{h(0)} \log x \\
& + R_{U^\varepsilon} (x) 
\end{split}
\end{equation}
where 
\begin{equation}
%\label{eq:}
\begin{split}
\vert R_{U^\varepsilon} (x) \vert \le C \big[ \varepsilon^{1/2 - \alpha}/ {\mathfrak h}(x) + \varepsilon^{2\alpha}   \big] \ .
\end{split}
\end{equation}
\end{lemma}

\begin{proof} 

Let $\delta>0$ such that $c(\delta)>0$. For $y \in [0,\delta]$, we want to approximate $\omega^\varepsilon (y)$ by 
\begin{equation}
\varepsilon f(0) + y h(0) \ .
\end{equation}

We have that 
\begin{equation}
\label{eq:0delta}
\begin{split}
\left\vert \int_0^\delta \frac{dy}{\omega^\varepsilon (y)} - \int_0^\delta \frac{dy}{\varepsilon f(0) + y h(0)}\right\vert &= \left\vert \int_0^\delta \cfrac{\varepsilon (f(y)- f(0) +y (h(y) -h(0))}{[\varepsilon f(y) + y h(y)][\varepsilon f(0) + y h(0)]} \, dy\right\vert \\
&\lesssim \int_0^\delta \cfrac{\varepsilon y + y^2}{(\varepsilon c + yc)^2} \, dy = c^{-2} \varepsilon  \int_0^{\delta/\varepsilon} \cfrac{ z + z^2}{(1+z)^2} \, dz \\
&\le c^{-2} \ \left[ \varepsilon \log (1+ \delta/\varepsilon) + \delta \right]      \ .   
\end{split}
\end{equation}
Since
\begin{equation}
\int_0^\delta \frac{dy}{\varepsilon f(0) + y h(0)} = \frac{1}{h(0)} \, \log\left(1 + \frac{\delta h(0)}{\varepsilon f(0)}\right) =- \frac{1}{h(0)} \, \log \varepsilon + \frac{1}{h(0)} \log\left(\frac{\delta h(0)}{f(0)} \, + \, \varepsilon \right) \ ,
\end{equation}
we get that 
\begin{equation}
%\label{eq:}
\begin{split}
\int_0^\delta \frac{dy}{\omega^\varepsilon (y)} &= - \frac{1}{h(0)} \, \log \varepsilon + \frac{1}{h(0)} \log\left(\frac{h(0) \delta}{f(0)} + \varepsilon \right)  \\ 
&\quad \quad  +  O (- c^{-2} (\delta) \  \varepsilon \log (1+ \delta/\varepsilon))  \ + \ O(c^{-2} (\delta) \, \delta) \\
&=  - \frac{1}{h(0)} \, \log \varepsilon + \frac{1}{h(0)} \log\left(\frac{h(0)}{f(0)} \right) + \frac{1}{h(0)} \log \delta \\
&+ \frac{1}{h(0)} \, \log \left( 1+ \frac{f(0)}{h(0)} \frac{\varepsilon}{\delta} \right)  +  O (- c^{-2} (\delta) \  \varepsilon \log (1+ \delta/\varepsilon))  \ + \ O(c^{-2} (\delta)\,  \delta)
\end{split}
\end{equation}

\bigskip

On the interval $[\delta,x]$, we have
\begin{equation}
%\label{eq:}
\begin{split}
y h(y) \ge \delta {\mathfrak h}(x), \quad \varepsilon |f(y)| \le \Vert F\Vert_{\infty} 
\end{split}
\end{equation}
so that 
\begin{equation}
%\label{eq:}
\begin{split}
\forall y \in [\delta, x], \quad \varepsilon f(y)+yh(y) \ge \delta {\mathfrak h}(x) - \varepsilon \Vert F \Vert_\infty \ . 
\end{split}
\end{equation}
It follows that if $\varepsilon \le \delta {\mathfrak h}(x)/ (2 \Vert F \Vert_{\infty})$, we have 
\begin{equation}
%\label{eq:}
\begin{split}
\forall y \in [\delta, x], \quad \varepsilon f(y)+yh(y) \ge \delta {\mathfrak h}(x) /2 
\end{split}
\end{equation}
and thus
\begin{equation}
%\label{eq:}
\begin{split}
&\left\vert \int_{\delta}^x  \frac{dy}{\varepsilon f(y) +y h(y)} - \int_\delta^x \frac{dy}{y h(y)}    \right\vert =\left\vert \int_{\delta}^x  \frac{\varepsilon f(y)}{ y h(y) (\varepsilon f(y) +y h(y))} \, dy    \right\vert \\
&\le 2 \varepsilon \cfrac{ \Vert F\Vert_{\infty}}{\delta^2 {\mathfrak h}^2 (x)}  \ .
\end{split}
\end{equation}
We have also that
\begin{equation}
%\label{eq:}
\begin{split}
\int_{0}^x \left\vert \frac{1}{y h(y)} -\frac{1}{h(0)y} \right\vert \, dy \lesssim \cfrac{x}{{\mathfrak h}(x)} < \infty \ .
\end{split}
\end{equation}
Finally, we observe that 
\begin{equation}
%\label{eq:}
\begin{split}
\int_{\delta}^x \frac{dy}{h(0)y} = \cfrac{1}{h(0)} \, [ \log x - \log \delta] \ . 
\end{split}
\end{equation}
Collecting all these informations together, and observing that the terms in $\log \delta$ cancels exactly we conclude that is  if $\varepsilon \le \delta {\mathfrak h}(x)/ (2 \Vert F \Vert_{\infty})$ then 
\begin{equation}
%\label{}
\begin{split}
U^\varepsilon(x) &= -\frac{1}{h(0)} \log \varepsilon +  \frac{1}{h(0)}  \log \Big(\frac{h(0)}{f(0)} \Big) +  \int_{0}^x \left[ \frac{1}{y h(y)} -\frac{1}{h(0)y} \right] \, dy  +  \frac{1}{h(0)} \log x \\
& + \log \left( 1+ \frac{f(0)}{h(0)} \frac{\varepsilon}{\delta} \right)  +   O (- c^{-2} (\delta) \  \varepsilon \log (1+ \delta/\varepsilon)) + O(c^{-2} (\delta) \delta) + O( \varepsilon \delta^{-2} {\mathfrak h}^{-2} (x)) \ .
\end{split}
\end{equation}
In the previous expansion the $O(\cdot)$ terms are uniform in $x$ and in $\varepsilon$. In particular $x$ could depend on $\varepsilon$. Since $\lim_{\delta \to 0} c(\delta) = \inf(f(0), h(0))>0$ and $\lim_{x \to 1^-} {\mathfrak h}(x) =0$,  we can not choose $\delta:=\delta(\varepsilon)$ arbitrarily going to zero or $x$ going to one as $\varepsilon$ goes to $0$. By choosing $\delta = 2 \Vert F \Vert_{\infty}  \varepsilon^{1/2-\alpha}/ {\mathfrak h}(x)$, with $\alpha \in (0,1/2)$, we get the lemma since the condition $(\delta {\mathfrak h}(x))/(2\Vert F \Vert_{\infty})= \varepsilon^{1/2- \alpha} \ge \varepsilon$ and
\begin{equation}
%\label{eq:}
\begin{split}
&\vert \log \left( 1+ \frac{f(0)}{h(0)} \frac{\varepsilon}{\delta} \right) \vert \lesssim \frac{\varepsilon}{\delta} \lesssim \varepsilon^{1/2+\delta} {\mathfrak h}(x) \ , \\
&\vert - c^{-2} (\delta) \  \varepsilon \log (1+ \delta/\varepsilon)) \vert  \lesssim \delta =2 \Vert F \Vert_{\infty}  \varepsilon^{1/2-\alpha}/ {\mathfrak h}(x) \ , \\
&\varepsilon \delta^{-2} {\mathfrak h}^{-1} (x) = (2 \Vert F \Vert_{\infty})^{-2} \, {\mathfrak h} (x) \, \varepsilon^{2\alpha} \ .
\end{split}
\end{equation}
 \end{proof}

%We conclude in particular that if $c$ is independent of $\varepsilon$, we have
%\begin{equation}
%%\label{eq:}
%\begin{split}
%\tau_c= -\frac{1}{h(0)} \varepsilon  \log \varepsilon  + \gamma (c) \varepsilon + o (\varepsilon) 
%\end{split}
%\end{equation}
%where 
%\begin{equation}
%%\label{eq:}
%\begin{split}
%\gamma (c) = \frac{1}{h(0)}  \log \left(\frac{h(0)}{f(0)} \right) +  \int_{0}^c \left[ \frac{1}{y h(y)} -\frac{1}{h(0)y} \right] \, dy  +  \frac{1}{h(0)} \log c \ .
%\end{split}
%\end{equation}
%
%In the case $f(y)=1/2 -y$, $h(y)=1-y$, this expansion coincides with the exact computation. 

\begin{lemma}
\label{lem:expV}
There exists a constant $C>0$ such that for any $\alpha \in (0,1/2)$ and $x\in (0,1)$, the following holds
\begin{equation}
\label{eq:exp-Vvarepsilon}
\begin{split}
V^\varepsilon(x) &= - \log \varepsilon +  \log \Big(\frac{h(0)}{f(0)} \Big)  +  \log x + R_{V^\varepsilon} (x) 
\end{split}
\end{equation}
where 
\begin{equation}
%\label{eq:}
\begin{split}
\vert R_{V^\varepsilon} (x) \vert \le C \big[ \varepsilon^{1/2 - \alpha}/{\mathfrak  h}(x) + \varepsilon^{2\alpha} {\mathfrak h}(x)  \big] \ .
\end{split}
\end{equation}
\end{lemma}

\begin{proof}
Recall the definitions of ${\mathfrak h}$ given in Eq. \eqref{eq:h} and of $c(\delta)$ given in Eq. \eqref{eq:cdelta}.\\

For $y \in [0,\delta]$, we want to approximate $\varepsilon f(y) + yh(y)$ by $\varepsilon f(0) + y h(0)$ and $h(y)$ by $h(0)$.

We have that 
\begin{equation}
%\label{eq:}
\begin{split}
\left\vert \int_0^\delta \frac{h(y)}{\varepsilon f(y) + yh(y)} \, dy - \int_0^\delta \frac{h(0)}{\varepsilon f(0) + y h(0)} \, dy \right\vert &= \left\vert \int_0^\delta \cfrac{\varepsilon (f(y)h(0) - f(0) h(y) )}{[\varepsilon f(y) + y h(y)][\varepsilon f(0) + y h(0)]} \, dy\right\vert \\
&\lesssim \int_0^\delta \cfrac{\varepsilon y}{(\varepsilon c + yc)^2} \, dy = \cfrac{1}{\varepsilon c^2} \int_0^{\delta/\varepsilon} \cfrac{\varepsilon^2 z}{(1+z)^2} \, dz \\
&\le c^{-2} \ \left[ \varepsilon \log (1+ \delta/\varepsilon)\right]      \ .   
\end{split}
\end{equation}
Since
\begin{equation}
\int_0^\delta \frac{h(0)}{\varepsilon f(0) + y h(0)} \, dy =  \log\left(1 + \frac{\delta h(0)}{\varepsilon f(0)}\right) =-  \log \varepsilon + \log\left(\frac{\delta h(0)}{f(0)} \, + \, \varepsilon \right) \ ,
\end{equation}
we get that 
\begin{equation}
%\label{eq:}
\begin{split}
\int_0^\delta  \frac{h(y)}{\varepsilon f(y) + yh(y)} \, dy & = -  \log \varepsilon + \log\left(\frac{h(0)}{f(0)}\right) \ + \ \log \delta\\
& \ + \ \log \left( 1+ \frac{f(0)}{h(0)} \frac{\varepsilon}{\delta} \right)  +  O (c^{-2} (\delta) \  \varepsilon \log (1+ \delta/\varepsilon))   \  .
\end{split}
\end{equation}

\bigskip 

Now, for the integral on $[\delta, x]$,  if $\varepsilon \le \delta {\mathfrak h}(x)/ (2 \Vert F \Vert_{\infty})$, by using the function $\mathfrak h$ defined by Eq. \eqref{eq:h}, we have
\begin{equation}
%\label{eq:}
\begin{split}
\forall y \in [\delta, x], \quad \varepsilon f(y)+yh(y) \ge \delta {\mathfrak h}(x) /2 
\end{split}
\end{equation}
and consequently 
\begin{equation}
%\label{eq:}
\begin{split}
&\left\vert \int_{\delta}^x  \frac{h(y)}{\varepsilon f(y) +y h(y)}\, dy  - \int_\delta^x \frac{dy}{y}    \right\vert =\left\vert \int_{\delta}^x  \frac{\varepsilon f(y)}{ y (\varepsilon f(y) +y h(y))} \, dy    \right\vert \le 2 \varepsilon \cfrac{ \Vert F\Vert_{\infty}}{\delta^2 {\mathfrak h} (x)}  \ .
\end{split}
\end{equation}
As for $U^\varepsilon$, we observe that $\int_{\delta}^x y^{-1} dy =\log x -\log \delta$, so that the $\log \delta$ terms cancel exactly and we get that, if $\varepsilon \le \delta {\mathfrak h}(x)/ (2 \Vert F \Vert_{\infty})$, then 
\begin{equation}
%\label{eq:}
\begin{split}
V^\varepsilon(x) &= - \log \varepsilon +  \log \Big(\frac{h(0)}{f(0)} \Big)  +  \log x \\
&+ \log \left( 1+ \frac{f(0)}{h(0)} \frac{\varepsilon}{\delta} \right)  +   O (- c^{-2} (\delta) \  \varepsilon \log (1+ \delta/\varepsilon)) + O( \varepsilon \delta^{-2} {\mathfrak h}^{-1} (x)) \ .
\end{split}
\end{equation}

By choosing $\delta = 2 \Vert F \Vert_{\infty}  \varepsilon^{1/2-\alpha}/ {\mathfrak h}(x)$, with $\alpha \in (0,1/2)$, we get the lemma since the condition $(\delta {\mathfrak h}(x))/(2\Vert F \Vert_{\infty})= \varepsilon^{1/2- \alpha} \ge \varepsilon$ and
\begin{equation}
%\label{eq:}
\begin{split}
&\left\vert \log \left( 1+ \frac{f(0)}{h(0)} \frac{\varepsilon}{\delta} \right) \right\vert \lesssim \frac{\varepsilon}{\delta} \lesssim \varepsilon^{1/2+\alpha} {\mathfrak h}(x) \ , \\
& \vert - c^{-2} (\delta) \  \varepsilon \log (1+ \delta/\varepsilon)) \vert  \lesssim \delta =2 \Vert F \Vert_{\infty}  \varepsilon^{1/2-\alpha}/ {\mathfrak h}(x) \ , \\
&\varepsilon \delta^{-2} {\mathfrak h}^{-2} (x) = (2 \Vert F \Vert_{\infty})^{-2} \, {\mathfrak h}(x) \, \varepsilon^{2\alpha} \ .
\end{split}
\end{equation}

\end{proof}

\begin{lemma}
\label{lem:Iepsilon}
Let $I^\varepsilon$ be defined by 
\begin{equation}
%\label{eq:}
\begin{split}
I^\varepsilon := \int_0^{T^\varepsilon_*}  \, r \,  (-\dot{\mu}_r^\varepsilon)  \, dr  =  \varepsilon  \int_0^{y_*^\varepsilon} \frac{h(x) }{\varepsilon f(x) + xh(x)} \, U^\varepsilon (x) \, e^{-V^\varepsilon (x)} \, dx  \ .
\end{split}
\end{equation}
There exists a constant $C>0$ such that for $\varepsilon>0$ in a neighborood of $0$, 
\begin{equation}
%\label{eq:}
\begin{split}
I^\varepsilon \ge C \varepsilon  \ .
\end{split}
\end{equation}
\end{lemma}

\begin{proof}
We wish to estimate from below the integral $I^\varepsilon$ rewritten as 
\begin{equation}
I^\varepsilon = \varepsilon \int_0^{y_*^\varepsilon} \frac{h(x)}{\omega^\varepsilon(x)} \,  U^\varepsilon(x) \,   e^{-V^\varepsilon(x)}\,   dx \ ,
\end{equation}
where we recall that $\omega^\varepsilon(x) = \varepsilon f(x) + x h(x)$ and $\lim_{\varepsilon \to 0} y_*^\varepsilon =1$.
Since $f$ and $h$ are continuous and $f(0) > 0$, $h(0) > 0$, there exists $\delta > 0$ such that for all $x \in [0,\delta]$,
\begin{equation}
2 f(0) \ge f(x) \ge \frac{f(0)}{2}, \quad 2 h(0) \ge h(x) \ge \frac{h(0)}{2} \ .
\end{equation}
For sufficiently small $\varepsilon > 0$ such that $\varepsilon \leq \delta$, we restrict the integration in $I^\varepsilon$ to $[0, \varepsilon]$. On $[0, \varepsilon]$, we have the following estimates:
\begin{equation}
\omega^\varepsilon(x) = \varepsilon f(x) + x h(x) \le 2\varepsilon (f(0) + h(0)),
\end{equation}
so that
\begin{equation}
\frac{1}{\omega^\varepsilon(x)} \geq \frac{1}{2\varepsilon (f(0) + h(0))} \ .
\end{equation}
Next,
\begin{equation}
U^\varepsilon(x) = \int_0^x \frac{dy}{\omega^\varepsilon(y)} \geq \int_0^x \frac{dy}{2\varepsilon (f(0) + h(0))} = \frac{x}{2\varepsilon (f(0) + h(0))} \ .
\end{equation}
Also, for $x\le \varepsilon$, 
\begin{equation}
V^\varepsilon(x) = \int_0^x \frac{h(y)}{\omega^\varepsilon(y)}  dy \leq \int_0^x \frac{2h(0)}{\varepsilon \cdot f(0)/2}  dy = \frac{4h(0) x}{\varepsilon f(0)} \leq \frac{4h(0)}{f(0)} \ ,
\end{equation}
so that
\begin{equation}
e^{-V^\varepsilon(x)} \geq e^{-4h(0)/f(0)} \ .
\end{equation}
Combining these estimates, the integrand satisfies
\begin{equation}
\varepsilon \, \frac{h(x)}{\omega^\varepsilon(x)} \,  U^\varepsilon(x)  \,  e^{-V^\varepsilon(x)} \geq \varepsilon \cdot \frac{h(0)/2}{2\varepsilon (f(0) + h(0))} \cdot \frac{x}{2\varepsilon (f(0) + h(0))} \cdot e^{-4h(0)/f(0)} = K \frac{x}{\varepsilon} \ ,
\end{equation}
where
\begin{equation}
K = \frac{h(0) e^{-4h(0)/f(0)}}{8 (f(0) + h(0))^3} \ .
\end{equation}
Therefore,
\begin{equation}
I^\varepsilon \geq K \,  \int_0^\varepsilon \frac{x}{\varepsilon} \,  dx = \frac{K}{2} \, \varepsilon  \ .
\end{equation}
Thus, there exists a constant $C>0$ such that for sufficiently small $\varepsilon > 0$, 
\begin{equation}
I^\varepsilon \geq C\,  \varepsilon \ .
\end{equation}
\end{proof}

\begin{lemma}
\label{lem:Vyepsilonstar}
We have that 
\begin{equation}
%\label{eq:}
\begin{split}
\lim_{\varepsilon \to 0} e^{-V^\varepsilon (y_*^\varepsilon)} = 0 \ .
\end{split}
\end{equation}
\end{lemma}

\begin{proof}
Recall that \(y_*^\varepsilon = 1 - \varepsilon^\beta\) with \(\beta \in (0,1/2)\).  From Lemma \ref{lem:expV}, for any \(x \in (0,1)\) and \(\alpha \in (0,1/2)\),
\[
V^\varepsilon(x) = -\log \varepsilon + \log\left(\frac{h(0)}{f(0)}\right) + \log x + R_{V^\varepsilon}(x) \ ,
\]
with remainder satisfying
\[
|R_{V^\varepsilon}(x)| \leq C\bigl[\varepsilon^{1/2-\alpha}/\mathfrak{h}(x) + \varepsilon^{2\alpha}\mathfrak{h}(x)\bigr] \ ,
\]
where \(\mathfrak{h}(x) = \inf_{y\in[0,x]} h(y)\).

\bigskip

Take \(x = y_*^\varepsilon = 1 - \varepsilon^\beta\).   Since \(h\) is continuously differentiable, \(h(1)=0\) and $h'(1)<0$, there exists \(c>0\) such that for sufficiently small \(\varepsilon\),
\[
c \varepsilon^\beta \le \mathfrak{h}(y_*^\varepsilon) \le c^{-1} \varepsilon^\beta \ .
\]
Hence, there exists a constant $C>0$ such that
\[
|R_{V^\varepsilon}(y_*^\varepsilon)| \leq C\bigl[\varepsilon^{1/2-\alpha - \beta} + \varepsilon^{2\alpha+\beta}\bigr] \ .
\]
Choose \(\alpha = \beta/2\) (note that \(\beta<1/2\) implies \(\alpha\in(0,1/4)\)). Then
\[
\varepsilon^{1/2-\alpha-\beta} = \varepsilon^{1/2 - 3\beta/2}, \qquad \varepsilon^{2\alpha+\beta} = \varepsilon^{2\beta} \ .
\]
For \(\beta < 1/3\) (which is compatible with \(\beta\in(0,1/2)\)), both exponents are positive, so that 

\[
\lim_{\varepsilon \to 0} R_{V^\varepsilon}(y_*^\varepsilon) =0 \ .
\]

\bigskip

Now,

\[
V^\varepsilon(y_*^\varepsilon) = -\log \varepsilon + \log\left(\frac{h(0)}{f(0)}\right) + \log(1-\varepsilon^\beta) + o(1).
\]

Since \(\log(1-\varepsilon^\beta) = O(1)\), we obtain

\[
V^\varepsilon(y_*^\varepsilon) = -\log \varepsilon + O(1) \xrightarrow[\varepsilon\to0]{} +\infty.
\]

\bigskip

Consequently,

\[
\exp\bigl(-V^\varepsilon(y_*^\varepsilon)\bigr) = \exp\bigl(\log\varepsilon + O(1)\bigr) = \varepsilon \cdot e^{O(1)} \xrightarrow[\varepsilon\to0]{} 0.
\]

\bigskip

Thus, \(\exp(-V^\varepsilon(y_*^\varepsilon)) \to 0\) as required.

\end{proof}

%%%%%%%%%%%%%%%%%%%%%%%%%%%%%%%%%%%%%%%%%%%%%%%%%%%%%%%%%%%%%%%%%%%%%%%

\section{Uniform convergence from Laplace transform convergence}

\begin{lemma}
\label{lem:Tauber}
Suppose $(h^\varepsilon)_{\varepsilon>0}$ and $h$ are non-negative continuous non-increasing functions from $[0,\infty)$ into $\mathbb R$ uniformly bounded by a constant $M>0$. If
\[
\forall \sigma>0, \quad \lim_{\varepsilon \to 0} \int_0^\infty e^{-\sigma s} h^\varepsilon (s) \, ds = \int_0^\infty e^{-\sigma s} h (s) \, ds \ ,
\]
then $(h^\varepsilon)_{\varepsilon> 0}$ converges to $h$ uniformly on compact subsets of $[0,\infty)$.
\end{lemma}

\begin{proof}
 
We proceed in several steps.

\bigskip

Since the family $(h^\varepsilon)_{\varepsilon>0}$ is uniformly bounded and consists of monotone (decreasing) functions, by Helly selection theorem \cite[Theorem 25.9]{B95}, every sequence $(\varepsilon_n)_{n \ge 0}$ with $\varepsilon_n \to 0$ has a subsequence $(\varepsilon_{n_k})_{k \ge 0}$ such that $(h^{\varepsilon_{n_k}})_{k\ge 0}$ converges pointwise to some function $\tilde{h}$ on $[0,\infty)$. Moreover, $\tilde{h}$ is non-increasing (as a pointwise limit of non-increasing functions), and $0 \le \tilde{h} \le M$.

\bigskip

Fix $\sigma > 0$. Since $\lim_{k \to \infty} h^{\varepsilon_{n_k}} = \tilde{h}$ pointwise and all functions are bounded by $M$, the dominated convergence theorem implies
\[
\lim_{k \to \infty} \int_0^\infty e^{-\sigma s} h^{\varepsilon_{n_k}}(s)  ds = \int_0^\infty e^{-\sigma s} \tilde{h}(s)  ds \ .
\]
By the hypothesis, the same integrals converge to $\int_0^\infty e^{-\sigma s} h(s)  ds$. Hence,
\[
\int_0^\infty e^{-\sigma s} \tilde{h}(s)  ds = \int_0^\infty e^{-\sigma s} h(s)  ds \quad \text{for all } \sigma > 0 \ .
\]
The Laplace transform is injective on bounded measurable functions, so $\tilde{h}(s) = h(s)$ for almost every $s$. Since $h$ and $\tilde{h}$ are both non-increasing and right-continuous (in fact, continuous by assumption), they are equal at all points. Thus $\tilde{h} = h$ everywhere.

\bigskip

We have shown: every subsequence of $(h^\varepsilon)_{\varepsilon >0}$ has a further subsequence converging pointwise to $h$. This implies the full family $(h^\varepsilon)_{\varepsilon>0}$ converges pointwise to $h$ as $\varepsilon$ vanishes.

\bigskip

Now, $h^\varepsilon$ and $h$ are continuous and non-increasing. On any compact interval $[0,T]$, pointwise convergence of monotone functions to a continuous limit implies uniform convergence by (second) Dini  theorem{\footnote{It is not clear that this theorem is really due to Dini  but it appears in \cite{PS98}.}}. Thus $(h^\varepsilon)_{\varepsilon>0}$ converges to $h$ uniformly on $[0,T]$. Since $T>0$ was arbitrary, the convergence is uniform on all compact subsets of $[0,\infty)$.

\end{proof}

\section*{Acknowledgements}
The article was prepared within the framework of the Basic Research Program at HSE University.

\bibliographystyle{alpha}  
\bibliography{references}

\end{document}